\documentclass[11pt,a4paper]{amsart}

\usepackage[margin=1in]{geometry}
\usepackage{amsmath,amssymb,amsthm,mathtools}
\usepackage{enumitem}
\usepackage{hyperref}
\usepackage{xcolor}
\usepackage{microtype}
\usepackage{comment}
\usepackage{mathrsfs}
\usepackage{soul}
\usepackage{graphicx}
\newcommand{\Schw}{\mathcal{S}}

\numberwithin{equation}{section}

\newcommand{\la}{\ensuremath{\langle}}
\newcommand{\ra}{\ensuremath{\rangle}}

\hypersetup{
    colorlinks=true,
    linkcolor=black,
    citecolor=black,
    urlcolor=blue
}

\setlist[itemize]{leftmargin=2em}
\setlist[enumerate]{leftmargin=2em}

\newcommand{\R}{\mathbb{R}}

\newcommand{\eps}{\varepsilon}

\newcommand{\norm}[2][]{\left\lVert #2\right\rVert_{#1}}

\newcommand\res{\mathrm{res}}
\newcommand\nr{\mathrm{nr}}
\newcommand\Id{\operatorname{Id}}
\newcommand\rcomp{\overline{\mathbb{R}^{3,1}}}
\newcommand\ang[1]{\langle #1 \rangle}
\newcommand\bang[1]{\big\langle #1 \big\rangle}
\newcommand\xib{{\overline{\xi}}}

\newcommand\mc[1]{\mathcal{#1}}

\DeclareMathOperator{\supp}{supp}

\theoremstyle{plain}
\newtheorem{theorem}{Theorem}[section]
\newtheorem{proposition}[theorem]{Proposition}
\newtheorem{lemma}[theorem]{Lemma}

\newtheorem{corollary}[theorem]{Corollary}
\newtheorem{conjecture}[theorem]{Conjecture}

\theoremstyle{definition}
\newtheorem{remark}[theorem]{Remark}
\newtheorem{definition}[theorem]{Definition}

\title[Scattering in spacetime cones]{Scattering in spacetime cones for the nonlinear Schr\"odinger equation}

\author{Andrew Hassell}
\email{Andrew.Hassell@anu.edu.au}
\author{Qiuye Jia}
\email{Qiuye.Jia@anu.edu.au}
\author{Xiaoxu Wu}
\email{Xiaoxu.Wu@anu.edu.au}
\address{Mathematical Sciences Institute, Australian National University, Canberra, ACT 2600 AUSTRALIA}
\keywords{Nonlinear Schr\"odinger equation, soliton resolution conjecture, scattering in spacetime cones}

\begin{document}

\begin{abstract}
    We consider the nonlinear Schr\"odinger equation on $\R^{3,1}$, with a class of `mixed' nonlinearities that admit soliton solutions. We are interested in the long-time behaviour of solutions with finite, but possibly large, mass and/or energy. By a spacetime cone, we mean an open set in $\R^{3,1}_{x,t}$ that is invariant under dilation $(x, t) \mapsto (\alpha x, \alpha t)$ for all $\alpha > 0$.
    Our main result is as follows: let $u(x, t)$ be the solution with initial condition in $H^2(\R^3)$. Let $\mathscr{M}, \mathscr{E}$ be the mass and energy of $u$, and let $\eta = \eta(\mathscr{M}, \mathscr{E}) > 0$ be sufficiently small. Let $C_\epsilon$ be the complement of $N \geq 1$ of sufficiently thin spacetime cones $\{ (x, t) \mid |x-v_jt| \leq \epsilon t \}$, where $v_1, \dots, v_N$ are fixed velocities and $\epsilon$ is sufficiently small. If the $L^\infty$ norm of $u$ in $C_\epsilon$ is $\leq \eta$, we prove that $u$ scatters in $C_{2\epsilon} \subset C_\epsilon$. 
\end{abstract}

\maketitle

\tableofcontents

\section{Introduction}

Consider the nonlinear Schr\"odinger equation on $\R^{3,1}$:
\begin{equation}\label{eq:nls}
(D_t + \Delta) u + \mc{N}(|u|) u = 0, 
\end{equation}
where $D_t = -i \partial_t$ and $\Delta = -\sum_{j=1}^3 \partial^2_{x_j}$ is the positive Laplacian in $\R^3$. We shall specify the properties of the nonlinearity $\mc{N}(|u|)$ shortly, but for now we mention that it is energy-subcritical and of `mixed' type: negative (focusing) for small $|u|$, and positive (defocusing) for large $|u|$. In particular, the equation \eqref{eq:nls} will satisfy global well-posedness in $H^1(\R^3)$, and admit soliton solutions. See \eqref{eq:Ncondition 1} --- \eqref{eq:Ncondition 3} for the precise assumptions on the nonlinearity. 

We consider solutions with finite mass
\begin{equation}\label{eq:mass}
\mathscr{M} := \int |u(x,t)|^2 \, dx
\end{equation}
and energy 
\begin{equation}\label{eq:energy}
\mathscr{E} := \int \Big( \frac{|\nabla u(x,t)|^2}{2} + \mc{E}(|u(x,t)|) \Big) \, dx, \qquad \mc{E}(b) = \int_0^b a \, \mc{N}(a) \, da;
\end{equation}
these are both conserved quantities, i.e. are constant in time for all $t \in \R$. 

\subsection{Scattering in spacetime cones}

We consider the scattering properties of such solutions \emph{locally in spacetime cones}. By a spacetime cone we mean a measurable subset of $\R^{3,1}$ invariant under dilation by positive numbers $\alpha$, namely under the map $(x,t) \mapsto (\alpha x, \alpha t)$. It is convenient to view this on the \emph{radial compactification of spacetime}, which consists of the union of $\R^{3,1}$ with a three-sphere `at infinity', one point for each ray in $\R^{3,1}$ through the origin. We denote this radial compactification by $\rcomp$. It is diffeomorphic to a 4-dimensional ball. It is customary to represent it with the time axis vertical. Then the boundary of $\rcomp$, which we will often call `spacetime infinity', consists of two open hemispheres, the `northern hemisphere', where $t = +\infty$, and the `southern hemisphere', where $t = -\infty$. Between them lies the `equator', which is the intersection between the boundary and any time-slice $\{ t = t_0 \}$, $t_0 \in \R$. See Figure~\ref{fig:radial compactification}. In this picture, there is a one-to-one correspondence between spacetime cones and measurable subsets of spacetime infinity. 

\begin{figure}[htbp]
    \centering 
    \includegraphics[width=0.7\textwidth]{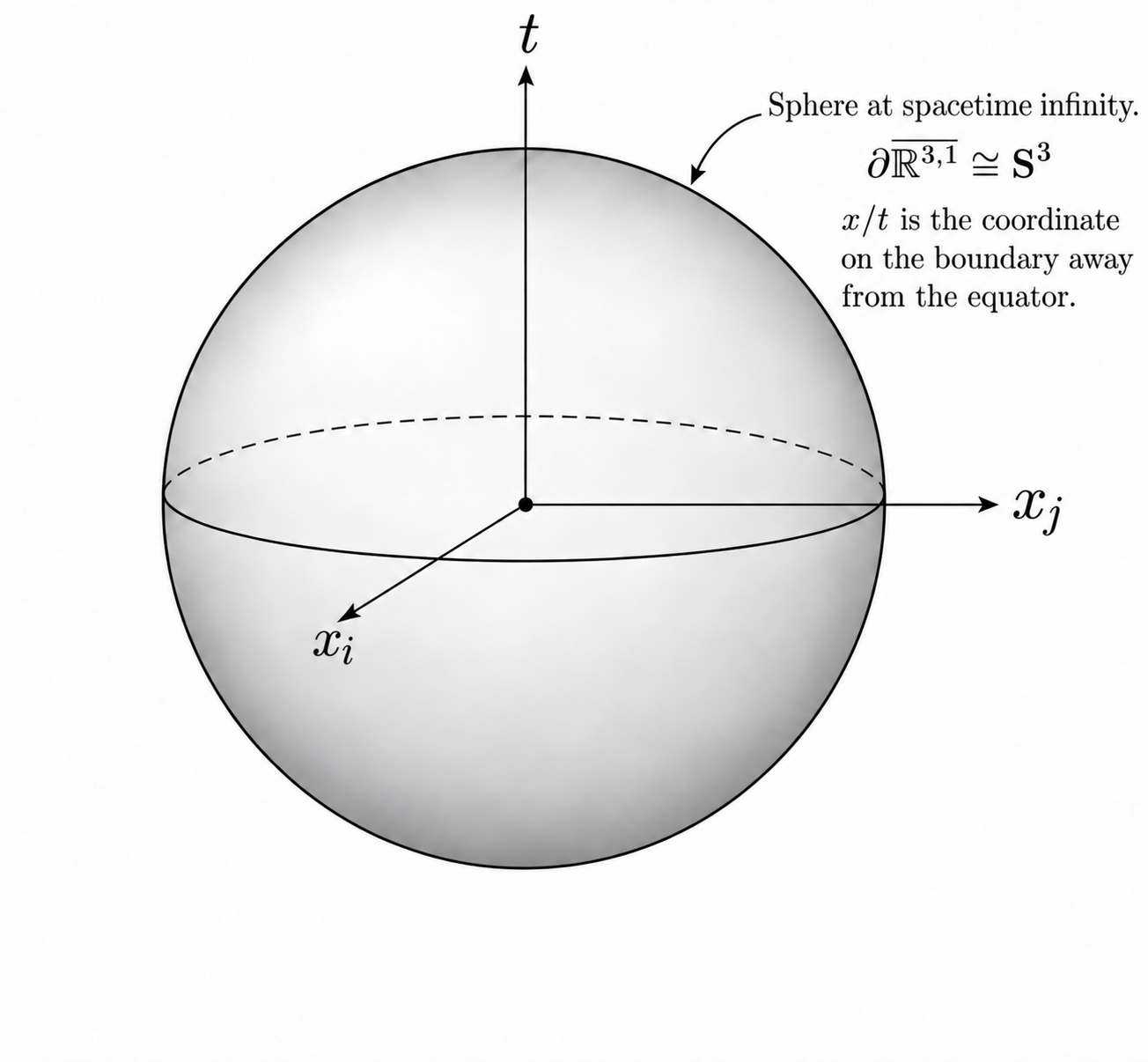} 
    \caption{The radial compactification of spacetime.}
    \label{fig:radial compactification} 
\end{figure}

The ratio $x/t$ is a smooth coordinate on both the northern and southern hemispheres, tending to infinity at the equator which is the boundary of each hemisphere. We think of this coordinate as an asymptotic velocity: any particle travelling with an asymptotic velocity $v$ will end up at position $v$ on the northern hemisphere, as $t \to +\infty$. Alternatively we can think of this as a `pseudoconformal' coordinate: the pseudoconformal transform of a \emph{scattering} solution $u(x,t)$ will have a continuous extension (in a suitable sense) to the boundary of spacetime. 

We now define the notion of scattering in a spacetime cone. 
\begin{definition}[Scattering in a spacetime cone]
Let $C$ be a spacetime cone, contained in $\{ t \geq 0 \}$.  
We say that a function $u(x, t)$ \emph{scatters} in $C$ (forward in time) if there exists $w \in L^2(\R^3)$ (possibly depending on $C$) such that 
\begin{equation}\label{eq:cone scattering defn} 
\lim_{t \to \infty} \| u(t) - e^{-it\Delta} w \|_{L^2(C_t)} = 0 
\end{equation}
where $C_{t} = \{ x \mid (x, t) \in C \} \subset \R^3$ is the time slice of the cone $C$ at time $t$. 
\end{definition}\label{def:cone scattering}

\begin{remark}\label{rem:cone scattering}
Although not immediately apparent from this definition (since $w$ is allowed to depend on $C$), if $u$ scatters in two open cones $C_1, C_2$, then it scatters in the union $C_1 \cup C_2$ of these cones --- see Lemma~\ref{lem:union of cones}.
\end{remark}

We will define the `thin cone of width $\epsilon$ in direction $v$' to be the spacetime cone
 \begin{equation}\label{eq:thin cone}
C_{v, \epsilon} =  \{ (x, t) \mid t > 0, \ |x-vt| \leq \epsilon t \}.
 \end{equation}
Our main result, Theorem~\ref{thm:main}, concerns scattering in a spacetime cone $C^\circ$ that consists of the complement of finitely many thin spacetime cones in the sense of \eqref{eq:thin cone}. It states that if we assume \emph{a priori} that $u$ is small in $L^\infty(C^\circ)$, then $u$ scatters in a slightly smaller $C \subset C^\circ$. This can be viewed as a localization of a well-known result of Rodnianski and Schlag to certain spacetime cones. Let us recall this result:

\begin{theorem}[A modified version of \cite{Rodnianski-Schalg-Strichartz}, Lemma 5.2]\label{thm:RS} Consider the linear Schr\"odinger equation with time-dependent potential:
\begin{equation}
    (D_t + \Delta + V(x,t))u(x,t) = 0, \quad u(0) = \phi \in L^2. 
\end{equation}
If $\| V \|_{L^\infty_t L^{3/2}_x}$ is sufficiently small, then $u$ satisfies the endpoint Strichartz estimate $\| u \|_{L^2_t L^6_x} < C \| \phi \|_{L^2_x}$, and hence $u$ scatters, in the sense that it is asymptotic to a free solution $e^{-it\Delta} \phi_+$, as $t \to \infty$. 
\end{theorem}

In our setting, the nonlinear potential $\mc{N}(|u|)$ will play the role of $V$, and smallness of $|u|$ will easily imply smallness of $V$ in $L^\infty_t L^{3/2}_x$ given our assumptions \eqref{eq:Ncondition 1} --- \eqref{eq:Ncondition 3} on $\mc{N}$. 
We also remark that in \cite{Rodnianski-Schalg-Strichartz}, additional assumptions on the potential are made. However, for Theorem~\ref{thm:RS}, only the smallness of $V$ in $L^\infty_t L^{3/2}_x$ is needed. In fact, the proof is so straightforward that we recall it here. 

\begin{proof}
By Duhamel's principle,
\begin{equation}\label{eq:Duhamel}
    \begin{aligned}
    u(t)=& e^{-it\Delta}u(0)+(-i)\int_0^t e^{-i(t-s)\Delta}V(s)u(s) ds.
    \end{aligned}
\end{equation}
Using the inhomogeneous Strichartz estimate, we find 
\begin{equation}
    \begin{aligned}
       \|\int_0^t e^{-i(t-s)\Delta}V(s)u(s) ds\|_{L^2_t([0,T))L^6_x} \leq &  C \| V(x,s) u(x, s) \|_{L^2_s([0,T)) L^{6/5}_x} \\ 
       \leq & C \| V(x,s)\|_{L^\infty_sL^{3/2}_x}\|u(x,s)\|_{L^2_s((0,T])L^6_x}.
    \end{aligned}
\end{equation}
This yields 
\begin{equation}
\|u(x,t)\|_{L^2_t((0,T])L^6_x}\leq C_1 \|u(0)\|_{L^2_x}+ C\| V(x,t)\|_{L^\infty_tL^{3/2}_x}\|u(x,t)\|_{L^2_t((0,T])L^6_x}.
\end{equation}
Provided $C \| V(x,t)\|_{L^\infty_tL^{3/2}_x} < 1$, we obtain  
\begin{equation}
\|u(x,t)\|_{L^2_t((0,T])L^6_x}\leq C_0\|u(0)\|_{L^2_x}
\end{equation}
for some fixed constant $C_0>0$. We then send $T\to \infty$ and yield 
\begin{equation}
\|u(x,t)\|_{L^2_t(\mathbb R^+)L^6_x}\leq C_0\|u(0)\|_{L^2_x},
\end{equation}
and similarly, 
\begin{equation}
\|u(x,t)\|_{L^2_t(\mathbb R^-)L^6_x}\leq C_0\|u(0)\|_{L^2_x}.
\end{equation}

To show that $u$ scatters, we need to show that $e^{it\Delta} u(x,t)$ converges as $t \to \infty$. Using \eqref{eq:Duhamel} it is enough to show that 
$$
\int_0^t e^{is\Delta}V(s)u(s) ds
$$
converges as $t \to \infty$. This follows from the adjoint Strichartz estimate and $Vu \in L^2_s L^{6/5}_x$, since then the tail 
$$
\int_{t_1}^{t_2} e^{is\Delta}V(s)u(s) ds
$$
converges to zero in $L^2$ as $t_1, t_2 \to \infty$. 

\end{proof}

\subsection{Main result}
Before stating the main result, we must specify the class of nonlinearities $\mc{N}(|u|)$. We assume the following:
\begin{equation}\label{eq:Ncondition 1}\begin{gathered}
    \text{$\mc{N}(|u|)$ is a $C^1$ function of $|u|^2$, which vanishes when $u=0$. In particular,} \\
    \text{$\mc{N}(|u|) = O(|u|^2)$ as $|u| \to 0$.}
\end{gathered}\end{equation}
\begin{equation}\begin{gathered}
    \text{There exists $b > 0$ such that $\mc{N}$ is focusing (that is, nonpositive) for $|u| \leq b$, } \\
    \text{and there exists $a_0\in(0,b)$ such that $\mathcal{N}(a_0)<0$,}\\
    \text { while $\mc{N}$ is defocusing (that is, nonnegative) for $|u| \geq b$. }
\end{gathered}\label{eq:Ncondition 2}\end{equation}
\begin{equation}\label{eq:Ncondition 3}
    \text{There exists $\kappa \in (0,1)$ such that $\mc{N}(|u|)$ is $O(|u|^{4-2\kappa})$ and $\mc{N}'(|u|)$ is $O(|u|^{3-2\kappa})$ as $|u| \to \infty$.}
\end{equation}

A good example to keep in mind is 
\begin{equation}\label{eq:exN}
 \mc{N}(|u|) = \frac{-|u|^2 + |u|^4}{(1 + |u|^2)^\kappa},      \end{equation}
for $0 < \kappa < 1$. The positivity of $\kappa$ ensures that this equation is energy-subcritical. When $\kappa = 0$ one has the cubic-quintic NLS studied by Killip, Oh, Pocovnicu and Visan \cite{Killip-Oh-Pocovnicu-Visan-Cubic-quintic}. Unfortunately, when $\kappa = 0$ in \eqref{eq:exN} we were not able to prove the uniform $L^\infty$ estimate in Proposition~\ref{prop:Linfty}. If one could prove such a uniform estimate, then the rest of our analysis would go through for the cubic-quintic equation. 

Our main result states that, for suitable spacetime cones, a small $L^\infty$ estimate implies scattering:

\begin{theorem}\label{thm:main}
 Let $\mc{N} = \mc{N}(|u|)$ be a nonlinearity satisfying \eqref{eq:Ncondition 1} -- \eqref{eq:Ncondition 3}. Let $u$ be a solution to the nonlinear Schr\"odinger equation \eqref{eq:nls} with $H^2(\R^3)$ initial data. Fix $N$ directions $v_1, \dots, v_N$, i.e. $N$ points at spacetime infinity, let $\epsilon > 0$ be sufficiently small, and define $C^\circ_\epsilon$ to be the complement of $N$ thin spacetime cones with direction $v_j$ and width $\epsilon$, as in \eqref{eq:thin cone}. 
Assume that $u$ is sufficiently small in $L^\infty$ in $C^\circ_\epsilon$, for all sufficiently large times:
 \begin{equation}\label{eq:cone condition 2}
 \| u \|_{L^\infty(C^\circ_\epsilon \cap \{ t \geq T \})} \leq \eta,
\end{equation}
where $\eta$ is sufficiently small depending on $u(0)$.

Then,  $u$ scatters in the (slightly smaller) spacetime cone $C^\circ_{2\epsilon}$, in the sense of Definition~\ref{def:cone scattering}. 
\end{theorem}

\begin{remark}
 The factor 2 in $C^\circ_{2\epsilon}$ compared to the cone $C^\circ_\epsilon$ is chosen for simplicity, and could be pushed closer to $1$. In our conjectural application of this result, discussed above, we expect the $L^\infty$ estimate could be proved outside of arbitrarily thin cones by taking $T$ large enough; then Theorem~\ref{thm:main} would show scattering outside of arbitrarily thin cones. For this reason, we have not pursued the question of reducing this factor of 2. 
\end{remark}

\subsection{Context: the soliton resolution conjecture}
We now discuss the context of Theorem~\ref{thm:main}, which is an approach to the soliton resolution conjecture that seems promising to the authors. 
Suppose that we have a solution $u$ of \eqref{eq:nls} with $u(x, 0) \in \mc{S}$.
Let $\mathscr M$ and $\mathscr E$ be the conserved mass and energy of this solution. We shall show that $\| u(t) \|_{H^1_x}$ is bounded uniformly in time, so by Sobolev embedding, the $L^6$ norm of $u(t)$ is also uniformly bounded. It follows that $|u(t)|^2 + |u(t)|^4$ is uniformly bounded in $L^{3/2}_x$, and therefore --- thanks to \eqref{eq:Ncondition 3} --- the nonlinear potential $\mc{N}$ is in $L^\infty_t L^{3/2}_x$. If it happens that $\mc{N}$ is \emph{small} in this space, then the Rodnianski-Schlag result applies and tells us that $u$ scatters. This smallness, in turn, would be implied by the smallness of $\| u \|_{L^\infty_{t, x}}$ since, for example, $\| |u|^2\|_{L^{3/2}} \leq \| u \|_{L^2}^{4/3} \| u \|_{L^\infty}^{2/3}$. 
One can adopt the point of view that the obstruction to scattering is $u$ remaining large (that is, larger than some threshold $\eps$, below which Theorem~\ref{thm:RS} applies) in $L^\infty([T, \infty)_t) L^{\infty}_x$ for all $T > 0$. The main new idea of this article is to pursue this point of view \emph{locally in spacetime cones}  instead of globally in spacetime. This approach requires, for example, Strichartz estimates localized to spacetime cones; see Proposition~\ref{prop:u exterior decay}. 

Another property that we shall prove for our class of nonlinearities is that $\| u \|_{L^\infty_x}$ is uniformly bounded in time; moreover, $u$ is H\"older continuous with any exponent $\beta < 1/2$, uniformly in the sense that the H\"older seminorm
\begin{equation}\label{eq:Holder seminorm}
    \sup_{x \neq y} \frac{|u(x, t) - u(y, t)|}{|x-y|^\beta}
\end{equation}
is bounded uniformly in time $t$. So suppose that the solution $u$ is such that 
$$
\| u \|_{L^\infty([T, \infty)_t; L^{\infty}_x)} \geq \eps 
$$
for all $T > 0$. Then, for at least some times $t \geq T$, there exists $x_0$ with $|u(x_0, t)| \geq \eps'$ for any $\eps' < \eps$. Using the H\"older continuity, there is a radius $\rho = \rho(\eps')$ such that $|u(x,t)| \geq \eps'/2$ on $B(x_0, \rho)$. On this ball, the mass of $u(t)$ is at least $m(\eps')$. It follows, using a standard covering argument, that the set 
$$
B(t) := \{ x \in \R^3 \mid |u(x, t)| \geq \eps' \}
$$
is contained in a finite number (in fact at most $O(\mathscr{M})$-many) of balls of radius $2\rho$. Thus, we have control not only on the measure of $B(t)$, but also its `shape'. 

The following heuristic discussion motivates what we call the `weak soliton resolution conjecture'. Let us call a ball $B(x_0, 2\rho)$ where $|u(x_0, t)| \geq \eps'$ a \emph{bubble}, or more precisely a bubble\footnote{We shall also use the term `bubble' imprecisely in Sections~\ref{sec:onebubble} and \ref{sec:multibubble} to denote a set, such as a cone, containing points where $|u|$ is not assumed to be small.} at time $t$. As we have just seen, the set $B(t)$ is contained in a union of $O(\mathscr{M})$ bubbles at time $t$. Assume, for a moment, that these bubbles behave somewhat like free particles, at least when widely separated (and close-together bubbles can be treated as a unit for the purposes of this discussion): that is, they have a trajectory that, asymptotically, has constant velocity, say $v$. Then, their position could be tracked in spacetime, and would tend to the point at spacetime infinity given by $x/t = v$. At all points of spacetime infinity other than the finite number of points to which bubbles converge, we could expect that we have scattering behaviour. Thus, we make the following conjecture which is a slight reformulation of the soliton resolution conjecture: 

\begin{conjecture}[Weak soliton resolution conjecture] Given a solution to \eqref{eq:nls} with suitable (e.g. $H^2$) initial data, there are finitely many points on the spacetime boundary such that, outside any neighbourhood of these points, the solution scatters. 
\label{conj:weak-soliton-resolution}
\end{conjecture}

\begin{remark}
 A neighborhood of a point on the spacetime boundary, with $x/t = v$, can be viewed (loosely speaking) as the interior of  a truncated thin cone in the direction $v$. By a truncated thin cone we mean $C_{v, \epsilon} \cap \{ t \geq T \}$ for $C_{v, \epsilon}$ defined by \eqref{eq:thin cone} and some $T > 0$. For $\epsilon$ sufficiently small and $T$ sufficiently large, this will be contained in any given neighbourhood of the point $\{ x/t = v, 1/t = 0 \}$ at spacetime infinity. 
\end{remark}

\begin{remark}
    The `weak' aspect of this conjecture is that it makes no attempt to specify what sort of asymptotic behaviour is occurring near the finitely many points (or even if there is an asymptotic). The only claim that the conjecture makes about the solution near such points is that the non-scattering behaviour spreads sub-linearly in time, otherwise it could not be confined to an arbitrarily small neighbourhood of a point on the spacetime boundary. Note that this property alone is sufficient to distinguish it from scattering behaviour, as the only scattering solution that spreads sub-linearly in time is the zero solution. 
\end{remark}

The weak soliton resolution conjecture could be proved if the following two assertions can be established:
\begin{itemize}
    \item[(I)] Show that there are finitely many velocities $v_1, \dots, v_N$, such that
for each $\epsilon > 0$, there exists $T > 0$ such that the solution $u(t)$ is small in $L^{\infty}_{t,x}$, uniformly for all $t \geq T$,  on the complement of $N$ thin cones of width $\epsilon$ around each direction $v_j$ (as defined in \eqref{eq:thin cone}). In other words, show that the hypothesis \eqref{eq:cone condition 2} is satisfied for $T$ large enough.

\item[(II)] Assume that \eqref{eq:cone condition 2} is satisfied for sufficiently small $\epsilon > 0$. Prove that $u(t)$ scatters in a slightly smaller cone. 
\end{itemize}

In this article, Theorem~\ref{thm:main} proves assertion (II). Addressing assertion (I) is beyond the scope of this article, but we plan to tackle it in the future. Assertion (I) is proved in \cite{jia2026NLS-largedata} in a different context; see comments below.

\subsection{Literature review}
The literature on nonlinear Schr\"odinger equations is vast and we make no attempt to survey it here. We only mention a few results that are closely related to the present study. 

Broadly speaking, for a suitable class of focusing or mixed nonlinear Schr\"odinger equations, there is a ground state, or more precisely a one-parameter family of ground states, that minimize some quantity $\mathscr{Q}$  related to the natural scaling (exact or asymptotic) of the equation. For example, for a mass critical equation, this would be the minimal mass of a soliton, while for an energy-critical equation, it would be $\dot H^1$-norm. Then, in a variety of situations it has been shown that `below the ground state' (in terms of $\mathscr{Q}$), one only sees scattering: for example, see \cite{KenigMerle2006, DuycHolmRoud2008, DodsonMurphy2017}. The asymptotics of solutions that are close to a ground state has been studied intensively: modulational stability, and the existence of a centre manifold, have been shown in many cases \cite{Schlag2009, NakanishiSchlagBook2011}. In some cases, detailed information on the behaviour of all  solutions with $\mathscr{Q}$ no larger than the ground state has been obtained \cite{NakanishiSchlagBook2011, Dodson2023, Dodson2024}. 

Soliton resolution is concerned with the asymptotic description of \emph{all} solutions --- or at least a dense subset --- of the energy space. Little is known about this question for most nonlinear Schr\"odinger equations, other than the completely integrable 1D case. 
For several completely integrable one-dimensional equations, including cubic 1D nonlinear Schr\"odinger, this has been achieved via the inverse-scattering, Riemann-Hilbert or Lax pair method 
\cite{BorgJenkMcLa2018, DNLS-sol-res, Sol-res-mKdV, soliton-res-Benjamin--Ono}. 
It has also recently been shown for the equivariant self-dual Chern-Simons-Schr\"odinger equation \cite{KimKwonOh2025}. 

Another approach to soliton resolution uses the idea of concentration-compactness, which is used in a slightly different form in aforementioned works \cite{KenigMerle2006, DuycHolmRoud2008} and many others as well. 
Tao  \cite{Tao-large-radial, Tao2007Attractor} showed that there is a `(concentration-)compact attractor', to which solutions to the NLS, after removing the scattering part and modulo symmetry, always converge. However this argument only works, currently, in dimension $d \geq 5$. 

There has been considerable progress on the soliton resolution conjecture for nonlinear wave equations, particularly recently by Collot, Duyckaerts, Kenig, and  Merle; see for example \cite{DKM2023, CDKM2024, CDKM2026}. The survey article \cite{CDKMReview2025} describes recent progress and gives further references on this topic. There are also results on wave maps \cite{wave-map-sol-res}.

The intent of this article is to make progress towards soliton resolution for nonlinear Schr\"odinger equations in three dimensions;  more precisely, it is intended as progress toward Conjecture~\ref{conj:weak-soliton-resolution}, which is a weaker version of it. The third author, with Soffer, has proved for some classes of nonlinear Schr\"odinger equations that the solution breaks into a free scattering piece and a weakly localized piece that spreads sublinearly \cite{SofferWu:arXiv:2304.04245,SofferWu2026}. 
The second author recently proved Conjecture~\ref{conj:weak-soliton-resolution} when $d\geq5$ for $\dot{H}^{1/2}$-critical NLS \cite{jia2026NLS-largedata}. 
Tao's result \cite{Tao2007Attractor} and results concerning the scattering below the ground state \cite{DodsonMurphy2017,Guevara12,Fang-Xie-Cazenave-scattering-11} are the foundations for the results of  \cite{jia2026NLS-largedata}.

The class of nonlinearity we consider here is motivated by the work of Killip, Oh, Pocovnicu, and Vi\c{s}an \cite{Killip-Oh-Pocovnicu-Visan-Cubic-quintic} on the cubic-quintic equation, which is a very natural model for a mixed nonlinearity (focusing for small $|u|$ and defocusing for large $|u|$). However, as mentioned above, our results do not (yet) apply to this case as we are unable to treat the energy-critical case. We leave as an interesting open problem to prove Proposition~\ref{prop:Linfty} in this case, after which all subsequent results would follow. 

\subsection{Outline of this article}
In Section~\ref{sec:spacetime cones}, we give some basic theory for scattering in spacetime cones. In particular we show that if a solution $u$ scatters in two spacetime cones, then it scatters in the union of the two cones. This is a simple observation but conceptually important for our strategy in Sections~\ref{sec:onebubble} and \ref{sec:multibubble}. 

In Section~\ref{sec:prelim}, we prove the key uniform $L^\infty$ bound  (Proposition~\ref{prop:Linfty}), together with some immediate consequences. 

In Section~\ref{sec:perturbed ev} we prove some preliminary estimates for the linear Schr\"odinger equation with time-dependent potential, along the lines of Rodnianski-Schlag. We then prove a crucial improvement on the standard dispersive estimate on a spacetime cone in Lemma~\ref{lemma: UL-non-stationary}. The integrability of this estimate as $s \to t$ is crucial for proving the Strichartz estimate localized to a spacetime cone in Section~\ref{sec:Strichartz}. 

In Section~\ref{sec:onebubble} we prove Theorem~\ref{thm:main} in the case that $C^\circ_\epsilon$ consists of the complement of a single cone of the form $|x-vt| \leq ct$. 
Our strategy is to work with conic cutoffs that depend on the spatial variable $x$ only through a single coordinate $x \cdot e$ for some fixed unit vector $e$. We then decompose frequency into `resonant' and `nonresonant' parts, which is such that under free evolution, only the resonant part would contribute to scattering in that spacetime cone determined by the nonvanishing of the conic cutoff. For the resonant part, we express the NLS solution $u$ by the Duhamel formula integrating from $t$ to $\infty$, which gives a nontrivial scattering term involving the weak limit $u_+$ from Lemma~\ref{lem:weak-limit-u-plus}, while for the nonresonant part we integrate from $0$ to $t$. This choice leads to integrals that have a non-stationary phase, and thus tend to zero as $t \to \infty$. This part of the proof is represented visually in Figure~\ref{fig:nsp}. Finally, by varying the direction $e$ we obtain the conclusion of Theorem~\ref{thm:main}. 

In Section~\ref{sec:equator} we note that using the previous section we can give a quick proof that only scattering occurs in a neighbourhood of the equator at spacetime infinity. 

Finally in Section~\ref{sec:multibubble} we prove the general form of Theorem~\ref{thm:main}. Our proof in Section~\ref{sec:onebubble} has been designed to generalize as easily as possible to the multi-bubble case, yet we still need a new ingredient, as the non-resonant argument from Section~\ref{sec:onebubble} does not adapt in any obvious way to the multi-bubble setting. The solution is to use the propagation estimate from Proposition~\ref{prop:3d-propagation}, which allows us to decompose phase space more finely, in order to obtain the nonresonant estimate and close the proof. 


\section{Scattering in spacetime cones}\label{sec:spacetime cones}
We have already given the definition of scattering in a spacetime cone, in Definition~\ref{def:cone scattering}. Usually, instead of integrating over a spacetime cone $C$ as in Definition~\ref{def:cone scattering}, we will multiply $u(t) - e^{-it\Delta} w$ by a cutoff function that is equal to $1$ in the cone $C$. Let us define a \emph{conic cutoff} to be a smooth (away from the origin of spacetime) function $\chi$ that is homogeneous of degree zero in $(x,t)$, taking values in $[0,1]$. 

The support of $\chi$ is a closed spacetime cone. We define the \emph{cross-sectional support} of a conic cutoff $\chi$ to be the cross-section of its supporting cone. This cross-section can be viewed as a closed subset of spacetime infinity; as we only work with cones contained in the positive time half-space, $\{ t \geq 0 \}$, this cross-section is a closed subset of the union of the (open) northern hemisphere and the equator. Abusing notation slightly, we will say that the support of $\chi$ has \emph{compact cross-section} if the support is a compact subset of the (open) northern hemisphere. This is coordinatized by $x/t$, so the support of $\chi$ has compact cross-section if and only if $x/t$ is bounded on the support of $\chi$. In general, we only need to consider conic cutoff functions $\chi$ having one of the following two properties: (i) the support of $\chi$ has compact cross-section, or (ii) the support of $1 - \chi$ has compact cross-section. In case (ii), $\chi = 1$ whenever $|x/t|$ is sufficiently large.

As promised in the Introduction, we show that scattering on each of two cones $C_1$, $C_2$ implies scattering on the union of the two cones. 

\begin{lemma}\label{lem:union of cones}
Suppose that $C_1$ and $C_2$ are two spacetime cones such that the function $u(x, t)$ scatters in both $C_1$ and $C_2$. Then $u$ scatters in $C_1 \cup C_2$.  
\end{lemma}

\begin{proof}
Let $A_i$ denote the cross-section of the cone $C_i$, namely the set such that $C_i \cap \partial \rcomp = \{ x/t \in A_i \}$. We write $A_i/2$ for $\{ v/2 \mid v \in A_i \}$. 
 Let $w_i$, $i = 1, 2$, be such that 
 \begin{equation}\label{eq:limit identity}
     \lim_{t \to \infty} \| u(t) - e^{-it\Delta} w_i \|_{L^2(C_{i, t})} = 0. 
 \end{equation}
 We claim that this limit identity determines $\hat w_i$ on the set $A_i/2$. If $w_i \in \Schw$, then this follows from the stationary phase expansion. In fact, we have 
 \begin{equation}\label{eq:spe}
e^{-it\Delta} w_i = (2\pi)^{-3}\int e^{i(x \cdot \xi - t |\xi|^2)} \hat w_i(\xi) \, d\xi \sim (4\pi it)^{-3/2} e^{i|x|^2/4t} \hat w_i\big(\frac{x}{2t}\big) + O\Big(t^{-5/2}\big( \ang{\frac{x}{t}} \big)^{-N} \Big). 
 \end{equation}
The error term $O(t^{-5/2} \ang{x/t}^{-N})$ makes no contribution to \eqref{eq:limit identity}, so can be neglected when taking this limit. We see that, for $w_i \in \Schw$, the limit identity \eqref{eq:limit identity} determines $\hat w_i$ in the set $A_i/2$. Indeed, if $w_i$ and $w_i'$ were two functions such that the limit identity \eqref{eq:limit identity} held for both, then by subtraction we would obtain 
 $$
  \lim_{t \to \infty} \| e^{-it\Delta} w_i - e^{-it\Delta} w'_i \|_{L^2(C_{i, t})} = 0.
 $$
 But this function is $e^{-it\Delta}(w_i - w_i')$, and by direct calculation, using the stationary phase expansion and passing to the `pseudoconformal' variable $x/t$, we see that this limit is equal to  
 a constant times the $L^2$ norm of $\hat w_i - \hat w_i'$ on the set $A_i/2$. Thus we have $\hat w_i = \hat w_i'$ on $A_i/2$. 
 
 More generally, by using the density of Schwartz functions in $L^2$, and exploiting the unitarity of $e^{-it\Delta}$, we see that this identity holds true for all $w_i \in L^2$. A similar argument shows that the values of $\hat w_i$ outside the set $A_i/2$ are irrelevant; in particular, $w_i$ is non-unique, as one can modify $\hat w_i$ arbitrarily outside the set $A_i/2$. 

 Now suppose that $w_1 = w_2 = w$. Then it is obvious that 
 \begin{equation}\label{eq:limit identity w}
     \lim_{t \to \infty} \| u(t) - e^{-it\Delta} w \|_{L^2(C_{1, t} \cup C_{2,t})} = 0, 
 \end{equation}
 so we have scattering in the union $C_1 \cup C_2$. In general, we may not have $w_1 = w_2$ since these functions are non-unique: as we just observed, $\hat w_i$ is only determined in $A_i/2$. But this means that the Fourier transforms $\hat w_1, \hat w_2$  agree on $A_1/2 \cap A_2/2$. Therefore, we can modify $w_1$ outside the set $A_1/2$, and $w_2$ outside the set $A_2/2$, so that $w_1 = w_2$ after modification, and the limit identity \eqref{eq:limit identity} still holds for $i=1, 2$. As we have seen, taking $w$ to be this common modification, we deduce the limit identity \eqref{eq:limit identity w} on the union $C = C_1 \cup C_2$, so $u$ scatters on the union of two cones if it scatters on each.  
\end{proof}

To show scattering, it is helpful to observe that there is an clear candidate for the choice of  $w$ in Definition~\ref{def:cone scattering}, namely, the weak limit 
\begin{equation}\label{eq:weak lim}
u_+ := \textrm{w-}\lim_{t \to \infty} e^{it\Delta} u(t). 
\end{equation}
As is well known, this weak limit exists under rather general assumptions; see Lemma~\ref{lem:weak-limit-u-plus} in our setting. The following simple lemma provides the basic mechanism for proving scattering in spacetime cones: 

\begin{lemma} 
\label{lem:tail estimate}
Let $A(t)$ be uniformly bounded on $L^2_x$. Assume that
\begin{align}
\lim_{t\to\infty}\sup_{T_1\geq t}\|\int_t^{T_1}A(t)e^{-i(t-s)\Delta}\mc{N}(|u(s)|)u(s)\,ds\|_{L^2_x}=0. \label{eq:ext-Cauchy}
\end{align}
Then
\begin{align}
\|A(t)(u(t)-e^{-it\Delta}u_+)\|_{L^2_x}\longrightarrow0. \label{eq:ext-N-convergence}
\end{align}
\end{lemma}

\begin{proof}
For $T_1\geq t$, Duhamel's formula gives
\begin{align} 
A(t)u(t)-A(t)e^{-it\Delta}e^{iT_1\Delta}u(T_1)=i\int_t^{T_1}A(t)e^{-i(t-s)\Delta}\mc{N}(|u(s)|)u(s)\,ds. \label{eq:duhamel-t-T}
\end{align}
Because of the existence of the weak limit \eqref{eq:weak lim}, for each fixed $t$, we have 
\begin{align}
A(t)e^{-it\Delta}e^{iT_1\Delta}u(T_1)\rightharpoonup A(t)e^{-it\Delta}u_+ \quad\text{as }T_1\to\infty.
\end{align}
So the left hand side of \eqref{eq:duhamel-t-T} has a weak limit  in $L^2$ as $T_1 \to \infty$ for fixed $t$. Therefore the right hand side of \eqref{eq:duhamel-t-T} has a weak limit as $T_1\to\infty$. 
Taking the weak limit as $T_1 \to \infty$ in \eqref{eq:duhamel-t-T}, we obtain 
\begin{equation}
A(t)\big( u(t)-e^{-it\Delta}u_+ \big) =i \ \textrm{w-} \!\!\! \lim_{T_1 \to \infty} \int_t^{T_1}A(t)e^{-i(t-s)\Delta}\mc{N}(|u(s)|)u(s)\,ds. \label{eq:duhamel-t-T-2}
\end{equation}
Using  weak lower semicontinuity, the norm of the right hand side is bounded by 
$$
\sup_{T_1\geq t}\big\|\int_t^{T_1}A(t)e^{-i(t-s)\Delta}\mc{N}(|u(s)|)u(s)\,ds \big\|_{L^2_x}. 
$$
Then \eqref{eq:ext-Cauchy} implies that the right hand side tends to zero in $L^2$ as $t \to \infty$, proving \eqref{eq:ext-N-convergence}.
\end{proof}
If we take the operator $A(t)$ to be multiplication by a conic cutoff $\chi(\cdot, t)$, then \eqref{eq:ext-N-convergence} shows scattering on the spacetime cone $\{ \chi = 1 \}$. To obtain \eqref{eq:ext-Cauchy} we will prove a Strichartz estimate on a spacetime cone; see Proposition~\ref{prop:u exterior decay}.


\section{Preliminary estimates}\label{sec:prelim}
\label{sec:preliminary-est}

\begin{proposition}\label{prop:N basic properties}
The solution to \eqref{eq:nls} with nonlinear potential $\mc{N}(|u|)$ satisfying \eqref{eq:Ncondition 1} --  \eqref{eq:Ncondition 3} satisfies global well-posedness in $H^1(\R^3)$. Moreover, with an initial condition $u(0) \in H^2(\R^3)$, we have 
\begin{itemize}
    \item[(i)] The $H^1_x$-norm of $u(t)$ is bounded uniformly in time by a constant $C = C(\mathscr{M}, \mathscr{E})$; 
    \item[(ii)] The NLS \eqref{eq:nls} admits non-scattering solutions;
    \item[(iii)] The nonlinear term $\mc{N}(|u(t)|) u(t)$ satisfies 
\begin{equation}\label{eq:nonlinear Lp bound}
\| \mc{N}(|u(t)|) u(t) \|_{L^p} \leq C(\mathscr{M}, \mathscr{E}), \quad p \in [1, \frac{6}{5}]
\end{equation}
uniformly in time, and
\begin{equation}\label{eq:nonlinear Lp gradient bound}
\| \nabla (\mc{N}(|u(t)|) u(t)) \|_{L^p} \leq C(\mathscr{M}, \mathscr{E}) \| u(t) \|_{L^\infty}^{1-\kappa}, \quad p \in [1, \frac{6}{6-\kappa}] 
\end{equation}
where $\kappa$ is as in \eqref{eq:Ncondition 3}. 
\end{itemize}
\end{proposition}

\begin{remark} Item (ii) above is not required for any of our results. We include it to show that the main theorem is non-vacuous. Estimate \eqref{eq:nonlinear Lp gradient bound} will be used in a bootstrap argument to show that the $L^\infty_x$-norm of $u$ is uniformly bounded, and it is important that the power of $\| u \|_{L^\infty_x}$ is strictly smaller than $1$. 
\end{remark}

\begin{proof} 
As these statements are well-known and easily obtained, we shall be brief. Global well-posedness follows from the main result of \cite{CazWeis1988}. Item (i) follows from \eqref{eq:Ncondition 1} and   \eqref{eq:Ncondition 2}. We bound the kinetic energy term in the energy by 
\begin{equation}\begin{gathered}
   \frac1{2}  \int |\nabla u|^2 \, dx = \mathscr{E} - \int \mc{E}(|u|) \, dx \leq \mathscr{E} + c\int |u|^4 1_{|u|\leq b} + b^41_{|u|> b} \, dx \\
   \leq \mathscr{E} + c b^2 \int |u|^2  \, dx \\
   = \mathscr{E} + c b^2 \mathscr{M}, 
  \end{gathered} 
\end{equation}
with $b$ as in \eqref{eq:Ncondition 2}, since $\mc{N}(|u|) = O(|u|^2)$ and hence $\mathcal{E}(|u|) = O(|u|^4)$ as $|u| \to 0$.

For item (ii), the main result of \cite{Berestycki-Lions-ground-state} shows the existence of soliton type solutions.
Here we include a short proof for the existence of non-scattering solutions.
It suffices to show that there are initial conditions with negative energy: such solutions cannot scatter as scattering solutions always have nonnegative energy.

By assumption, we can choose $a_0 > 0$ such that $\mc{N}(a_0) < 0$ and $\mc{N}(a) \leq 0$ for $a \in [0, a_0]$, and 
 $\phi \in C_c^\infty$ such that $\phi(x) = a_0$ for $x \in B(0, 1)$ and $\phi(x) \in [0,a_0]$ for all $x \in \R^3$. Then, 
$$
\int \mc{E}(|\phi(\frac{x}{R})|) \, dx < 0 \text{ for all } R > 0.
$$
Moreover, the integral above scales as $R^3$ for $R \to \infty$, while $\| \nabla \phi( \cdot/R) \|_2^2$ scales as $R$. It follows that for sufficiently large $R$, the initial condition $\phi(\cdot/R)$ has negative energy. 

To prove \eqref{eq:nonlinear Lp bound}, we note that by (i) and Sobolev embedding, $\| u \|_r$ is uniformly bounded for $2 \leq r \leq 6$. The $L^p$ norm in (iii) is bounded by $C \| u \|_{3p}^3 + \| u \|_{(5-2\kappa)p}^{5-2\kappa}$, hence uniformly bounded since $3p$ and $(5-2\kappa)p$ are both in the range $[2, 6]$. 

To prove \eqref{eq:nonlinear Lp gradient bound}, we bound $\nabla (\mc{N}(|u|) u)$ pointwise by $C|\nabla u| \big( |u|^2 1_{|u| \leq b} + |u|^{4-2\kappa} 1_{|u| > b} \big)$. The bound is straightforward for the $|u|^2$ term so we only consider the $|u|^{4-2\kappa}$ term. We remove a factor $|u|^{1-\kappa}$ which we bound in $L^\infty$, and then bound $|\nabla u|$ in $L^2$, $|u|$ in $L^6$, and use H\"older's inequality. 
\end{proof}

\begin{proposition}\label{prop:Linfty} Assume that $u(0) \in H^2(\R^3)$. Then $u(t) \in L^\infty(\R^3)$ uniformly in time, i.e. $u \in L^\infty_t L^\infty_x$. The bound depends only on $\mathscr{E}$, $\mathscr{M}$ and $\| u(0) \|_{H^2}$. 
\end{proposition}

\begin{proof}
For this, we expand $u(t)$ using Duhamel's principle:
\begin{equation}
    u(t)=e^{-it\Delta}u(0)+(-i)\int_0^t e^{-i(t-s)\Delta} \mathcal N(|u(x,s)|)u(x,s)ds.
\end{equation}
By Sobolev embedding, we have 
\begin{equation}\label{eq:u_0 term}
    \|e^{-it\Delta}u(0)\|_{L^\infty}\leq C \|e^{-it\Delta}u(0)\|_{H^2} = C\| u(0)\|_{H^2}. 
\end{equation}
For the integral term, we split it into two terms (we only need to consider the case $t\geq 1$):
\begin{equation}
    (-i)\int_0^t e^{-i(t-s)\Delta} \mathcal N(|u(x,s)|)u(x,s)ds=u_{\mathrm{int},1}(t)+u_{\mathrm{int},2}(t),
\end{equation}
where 
\begin{equation}
    u_{\mathrm{int},1}(t):=(-i)\int_0^{t-1} e^{-i(t-s)\Delta} \mathcal N(|u(x,s)|)u(x,s)ds
\end{equation}
and
\begin{equation}
    u_{\mathrm{int},2}(t):=(-i)\int_{t-1}^t e^{-i(t-s)\Delta} \mathcal N(|u(x,s)|)u(x,s)ds.
\end{equation}
The $L^1$ -- $L^\infty$ decay estimate of the free flow together with \eqref{eq:nonlinear Lp bound} ($p=1$ case) yields 
\begin{equation}
    \|u_{\mathrm{int},1}(t)\|_{L^\infty}\leq C \int_0^{t-1}|t-s|^{-3/2}\|\mathcal N(|u(x,s)|)u(x,s)\|_{L^1}ds\leq 2C C(\mathscr{M}, \mathscr{E})
\end{equation}
for some positive constants $C$ and $C(\mathscr{M}, \mathscr{E})$. 

Next, we estimate $\|u_{\mathrm{int},2}(t)\|_{L^\infty}$ by exploiting the smoothness of $\mc{N}(|u|) u$, as given by \eqref{eq:nonlinear Lp gradient bound}. To be precise, we write 
\begin{equation}
    u_{\mathrm{int},2}(t)=\frac{\Delta+1}{\Delta+1}  u_{\mathrm{int},2}(t)=\frac{1}{\Delta+1}u_{\mathrm{int},2}(t)+\sum\limits_{j} \frac{D_{x_j}}{\Delta+1} D_{x_j}u_{\mathrm{int},2}(t).
\end{equation}
Let $\epsilon_p > 1/2$ be a constant that we shall choose depending on a Lebesgue exponent $p$ shortly. 
To estimate $\frac{D_{x_j}}{\Delta+1} D_{x_j}u_{\mathrm{int},2}(t)$, we split it into the high-frequency part and the low frequency part:
\begin{equation}
    \frac{D_{x_j}}{\Delta+1} D_{x_j}u_{\mathrm{int},2}(t)=u_{\mathrm{int},2,j,h}(t)+u_{\mathrm{int},2,j,l}(t),
\end{equation}
where, for some cutoff function $\chi$ equal to $1$ on $[0,1/2]$ and $0$ on $[1, \infty)$, 
\begin{equation}\label{eq:high freq}
   u_{\mathrm{int},2,j,h}(t):= (-i)\int_{t-1}^t \frac{D_{x_j}}{\Delta+1} (1-\chi)\big( \frac{\Delta}{|t-s|^{-2\epsilon_p}}\big) e^{-i(t-s)\Delta} D_{x_j}\mathcal N(|u(x,s)|)u(x,s)ds
\end{equation}
and
\begin{equation}\label{eq:low freq}
   u_{\mathrm{int},2,j,l}(t):= (-i)\int_{t-1}^t \frac{D_{x_j}}{\Delta+1} \chi\big( \frac{\Delta}{|t-s|^{-2\epsilon_p}}\big)e^{-i(t-s)\Delta} D_{x_j}\mathcal N(|u(x,s)|)u(x,s)ds.
\end{equation}
To estimate the high frequency piece \eqref{eq:high freq} we use the $L^1 \to L^\infty$ dispersive estimate for the free flow. Thus we need to bound the $L^\infty \to L^\infty$ norm of the operator $\frac{D_{j}}{\Delta+1} (1-\chi)( \frac{\Delta}{|t-s|^{-2\epsilon_p}})$. This is a Fourier multiplier and its norm is bounded by the $L^1$ norm of its convolution kernel. This is given by (with $A := |t-s|^{-\epsilon_p}$)
\begin{equation}
(2\pi)^{-3}  \int e^{ix \cdot \xi} \frac{\xi_j}{1 + |\xi|^2} (1-\chi)\big(\frac{|\xi|^2}{A^2} \big) \, d\xi = (2\pi)^{-3} A^2 \int e^{iAx \cdot \xib} \frac{\xib_j}{A^{-2} + |\xib|^2} (1-\chi)(|\xib|^2) \, d\xib.
\end{equation}
This is $A^2 f_A(Ax)$ where 
$$
f_A(x) := (2\pi)^{-3} \int e^{ix \cdot \xib} \frac{\xib_j}{A^{-2} + |\xib|^2} (1-\chi)(|\xib|^2) \, d\xib
$$
is the inverse Fourier transform of a classical symbol of order $-1$, depending smoothly and parametrically on $A^{-2}$ (note that the singularity at $\xib = 0$ has been removed due to the $1-\chi$ cutoff, so this parameter is smooth down to $A^{-2} = 0$, i.e. uniformly for $A \to \infty$). The function $f_A$ is smooth for $x \neq 0$, rapidly decreasing at infinity, uniformly in $A \geq 1$, and $O(|x|^{-2})$ as $x \to 0$, again uniformly in $A \geq 1$. It follows that $f_A$ is uniformly in $L^1$. By scaling we see that $\| A^2 f_A(Ax) \|_{L^1} = O(A^{-1})$. 
So we can estimate the $L^1$ norm of this convolution kernel by $C A^{-1} = C |t-s|^{\epsilon_p}$. Putting these estimates together we obtain, using \eqref{eq:nonlinear Lp gradient bound} with $p=1$,  
\begin{equation}
    \begin{aligned}
&\|\sum\limits_{j}u_{\mathrm{int},2,j,h}(t)\|_{L^\infty}\\
        \leq& \int_{t-1}^t \sum\limits_{j}\| \frac{D_{j}}{\Delta+1} (1-\chi)( \frac{\Delta}{|t-s|^{-2\epsilon_p}}) \|_{L^\infty\to L^\infty} \frac{C}{|t-s|^{3/2}}\| D_{x_j}\mathcal N(|u(x,s)|)u(x,s)\|_{L^1}ds\\
        \leq & \int_{t-1}^t C_2|t-s|^{\epsilon_p-3/2}\sum\limits_{j}\| D_{x_j}\mathcal N(|u(x,s)|)u(x,s)\|_{L^1}ds \\
        \leq & \frac{C_2}{\epsilon_p-1/2} C(\mathscr{M}, \mathscr{E}) 
       \sup_{ s \in [t-1,t]} \|u(s)\|_{L^\infty}^{1-\kappa} 
    \end{aligned}
\end{equation}
for some positive constant $C_2$. 

Next we estimate the low-frequency piece \eqref{eq:low freq}. 
We note that the $L^\infty_x$-norm of $u_{\mathrm{int},2,j,l}(t)$ can be bounded by the $L^1$-norm of its Fourier transform. We will use \eqref{eq:nonlinear Lp gradient bound} of Proposition~\ref{prop:N basic properties} for some $p > 1$ for the function $D_{x_j}\mathcal N(|u(x,s)|)u(x,s)$. By the Hausdorff-Young inequality its Fourier transform is in $L^{p'}$. Thus we need to have the other factor in $L^p$, so that its product lies in $L^1$. We compute 
\begin{equation}\label{eq:low energy}
\begin{aligned}
&\|\sum\limits_{j}u_{\mathrm{int},2,j,l}(t)\|_{L^\infty(\R^3_x)} \leq \|\sum\limits_{j}\hat u_{\mathrm{int},2,j,l}(t)\|_{L^1(\R^3_\xi)}\\
\leq& \frac{C_p}{(2\pi)^3} \int_{t-1}^t \sum\limits_j \| \frac{\xi_j}{|\xi|^2+1} \chi\big( \frac{|\xi|^2}{|t-s|^{-2\epsilon_p}} \big) \|_{L^p(\R^3_\xi)} \| D_j\mathcal N(|u(x,s)|)u(x,s)\|_{L^p(\R^3_x)} ds\\
    \leq & C_3C_p C(\mathscr{M}, \mathscr{E}) \int_{t-1}^t |t-s|^{-\epsilon_p(3-p)/p} \, ds \ \Big( \sup_{s \in [t-1, t]} \| u(s) \|_{L^\infty_x}^{1-\kappa} \Big)\\
    =& \frac{C_3C_p C(\mathscr{M}, \mathscr{E})}{1-\epsilon_p\cdot \frac{3-p}{p}} \ \Big(\sup_{s \in [t-1, t]} \| u(s) \|_{L^\infty_x}^{1-\kappa} \Big)
    \end{aligned}
\end{equation}
for some positive constants $C_3$ and $C_p$, provided that $\epsilon_p < p(3-p)^{-1}$. Noting that $p(3-p)^{-1} > 1/2$, we can choose $\epsilon_p$ in the range  
\begin{equation}
    \epsilon_p\in (1/2, p/(3-p))
\end{equation}
so that both conditions $\epsilon_p -1/2>0$ and $1-\epsilon_p\cdot \frac{3-p}{p}>0$ are satisfied. 
Following a similar, easier argument yields 
\begin{equation}
    \|\frac{1}{\Delta+1}u_{\mathrm{int},2}(t)\|_{L^\infty}\leq C_4 C(\mathscr{M}, \mathscr{E})
\end{equation}
for some constant $C_4>0$. Combining these estimates, we obtain
\begin{multline}\label{eq:Linfty combined}
   \sup\limits_{t \in [0, T]}\|u(t)\|_{L^\infty} \leq \frac{1}{(2\pi)^{3}}\| u(0) \|_{H^2}+2CC(\mathscr{M}, \mathscr{E})+\frac{C_2}{\epsilon_p-1/2}C(\mathscr{M}, \mathscr{E}) \Big(\sup_{t \in [0,T]} \| u(t) \|_{L^\infty_x}\Big)^{1-\kappa} \\ 
  +\frac{C_3 C(\mathscr{M}, \mathscr{E}) \Big(\sup_{t \in [0,T]} \| u(t) \|_{L^\infty_x}\Big)^{1-\kappa}}{1-\epsilon_p\cdot \frac{3-p}{p}}+C_4 C(\mathscr{M}, \mathscr{E}).
\end{multline}
For $H^2$ initial data, the solution $u(t)$ remains in $H^2(\R^3)$ for each $t$ \cite[Theorem 5.3.1]{Caz-book-2003}, hence the supremum of the $L^\infty$ norm over $t \in [0, T]$ is finite. As we have a smaller power of this supremum on the RHS, \eqref{eq:Linfty combined} shows the uniform boundedness of this supremum over all $T \in \R$. 
This completes the proof. 
\end{proof}

\begin{remark}\label{rem:ext} Thanks to Proposition~\ref{prop:Linfty}, we can extend the range of $p$ in \eqref{eq:nonlinear Lp bound} to all $p \in [1, \infty]$, and in \eqref{eq:nonlinear Lp gradient bound} to all $p \in [1,2]$. Moreover, the range in \eqref{eq:nonlinear Lp gradient bound} improves further to all $p \in [1, 6)$ once we have proved the next proposition. 
\end{remark}

\begin{proposition}  \label{prop:gradient-L2to6-bound}
The $L^q$ norm of the gradient, $\| \nabla u \|_{L^q_x}$ is uniformly bounded in time for $2 \leq q < 6$. 
\end{proposition}

\begin{proof}
We decompose the solution $u(t)$ as $u(t) = e^{-it\Delta} u(0) + u_{\mathrm{int}, 1}(t) + u_{\mathrm{int}, 2}(t)$ as in the previous proof
and prove that the gradient of each term is in $L^q$ for $2 \leq q < 6$. 

For the free term $e^{-it\Delta} u(0)$ this follows from Sobolev embedding. 

For the third term $u_{\mathrm{int},2}$ we use the $L^{q'} \to L^q$ mapping property of the free propagator, with a bound $|t-s|^{-\alpha}$ with $\alpha < 1$, thus integrable in $s$. 
 For this, we require that $\nabla \mc{N}(|u|) u$ is in $L^{q'}$ uniformly in time, as observed in Remark~\ref{rem:ext}. 

 For the second term, we prove that the gradient is in $L^2$ and in $L^\infty$. The $L^2$ gradient bound follows because it holds for  $u(t)$ and for the other two terms in the decomposition. 
And its spatial derivative is uniformly bounded in $L^\infty$, using the $L^1 \to L^\infty$ mapping property of the free propagator and \eqref{eq:nonlinear Lp gradient bound} with $p=1$. By interpolation, the gradient of this term is bounded in $L^q$ for $2 \leq q \leq \infty$. 
\end{proof}

\begin{corollary}\label{col: holdercontinue} The solution $u(t)$ is H\"older continuous with exponent $\beta$ for any $\beta < 1/2$, uniformly in time. In particular, the H\"older seminorm \eqref{eq:Holder seminorm} is uniformly bounded. 
\end{corollary}

\begin{proof} An immediate consequence of the previous Proposition and the Sobolev embedding theorem. 
\end{proof}

\begin{remark}\label{rem:nonlinear potential bounds cone} For later use, we observe that we can estimate suitable $L^p$ norms of $\mc{N}(|u|)$ and $\nabla \mc{N}(|u|)$, on a spacetime cone $C$ in terms of the $L^\infty$ norm of $u$ on $C$. In fact, for $1 < p \leq \infty$ we can estimate
\begin{equation}\label{eq:nonlinear pot cones}
\| \mc{N}(|u(t)|) \|_{L^p(C_t)} \leq C(\mathscr{M}, \mathscr{E}) \| u(t) \|_{L^\infty(C_t)}^\alpha, \quad p > 1, \quad \alpha = \alpha(p) > 0
\end{equation}
(where $C_t$ is the time slice of the cone $C$) and 
\begin{equation}\label{eq:gradient nonlinear pot cones}
\| \nabla \mc{N}(|u(t)|) \|_{L^q(C_t)} \leq C' \| u(t) \|_{L^\infty(C_t)}, \quad 2 \leq q < 6
\end{equation}
by using Proposition~\ref{prop:gradient-L2to6-bound}, Proposition~\ref{prop:Linfty}, and the pointwise estimate 
$$
|\nabla \mc{N}(|u|)| \leq |\nabla u| \Big( |u| + |u|^{3-2\eta} \Big).
$$
\end{remark}

\section{Estimates on the perturbed evolution}\label{sec:perturbed ev}
\label{sec:perturbed-evolution}

We prove some preliminary estimates for the spacetime bound of our solutions.
Let $U_L(t,s)$ be the evolution operator for the Schr\"odinger operator with real-valued time-dependent potential $V_L(x,t)$. That is, $U_L(t,s)$ satisfies 
\begin{equation} \label{eq:def-UL}
    i\partial_t U_L(t,s) = \left( \Delta + V_L(t) \right) U_L(t,s), \quad U_L(s, s) = \operatorname{Id}. 
\end{equation}
(When we apply this lemma, the potential $V_L(x,t)$ will be the cut-off nonlinear potential $\chi(x, t)\mc{N}(|u(x,t)|)$, where $\chi$ cuts off to a conic region where $\| u \|_{L^\infty}$ is small.) 

\begin{lemma}\label{lem:UL properties} Consider a real-valued potential function $V_L(x,t)$ such that
\begin{equation}
  \|V_L\|_{L^\infty_t L^1_x} = K_0 <\infty, \quad \|V_L\|_{L^\infty_t L^\infty_x} = \delta_0 \ll 1, 
\end{equation}
where $\delta_0$ is sufficiently small depending on $K_0$.
Then the evolution operator $U_L(t, s)$ satisfies 
\begin{itemize}
    \item[(i)] $L^2$ mass conservation: 
    \begin{equation}\label{eq: mass conservation}
     \| U_L(t,s) f \|_{L^2_x} = \| f \|_{L^2_x} \text{ for all $t, s$. }  
    \end{equation}
    \item[(ii)] Endpoint Strichartz estimate: for each $s \in \R$, 
    \begin{equation}\label{eq:UL endpt Strichartz}
  \| U_L(t, s) f \|_{L^2_t L^6_x} \leq C \| f \|_{L^2_x}      
    \end{equation}
 where $C$ is uniform in $s$. 
 \item[(iii)] Dispersive estimate:
 \begin{equation}\label{eq: UL local decay}
    \| U_L(t,s)f\|_{L^2_x+L^\infty_x} \lesssim \frac{1}{\langle t-s\rangle^{3/2}}\|f\|_{L^2_x\cap L^1_x}.
\end{equation}
Assume in addition that $\nabla_x V_L \in L^\infty_t (L^1 \cap L^3)_x$, with 
\begin{equation}
  \|\nabla_x V_L\|_{L^\infty_t L^1_x} = K_1 <\infty, \quad \|\nabla_x V_L\|_{L^\infty_t L^3_x} = \delta_1 \ll 1,
\end{equation}
where $\delta_1$ is sufficiently small depending on $K_0$ and $K_1$. 
Then we also have 
    \item[(iv)] Uniform boundedness on $H^1$:
    \begin{equation}\label{eq:H1 uniform bd}
 \| U_L(t,s) f \|_{H^1} \leq C \| f \|_{H^1},        
    \end{equation}
    where $C$ is uniform in $t, s$. 
    \item[(v)] 
    Multiplication by $V_L$ is a map from  $W^{1,2}_x + W^{1, \infty}_x$ to $W^{1,2}_x$ with small norm and into  $W^{1,1}_x$ with uniformly bounded norm.
    In addition, the following dispersive estimate holds: 
    \begin{equation}\label{eq: UL local decay deriv}
    \| U_L(t,s)f\|_{W^{1,2}_x + W^{1, \infty}_x} \lesssim \frac{1}{\langle t-s\rangle^{3/2}}\|f\|_{W^{1,2}_x \cap W^{1,1}_x}.
\end{equation}
\end{itemize}
\end{lemma}

\begin{proof}
We only need to consider times $t\geq 0$.
    Property (i), \eqref{eq: mass conservation}, is immediate from the self-adjointness of the time-dependent Schr\"odinger operator $\Delta + V_L$. 

 Property (ii), the endpoint Strichartz estimate, \eqref{eq:UL endpt Strichartz}, follows from the proof of \cite[Lemma~5.2]{Rodnianski-Schalg-Strichartz}, which only requires the smallness of the potential in $L^\infty_tL^{3/2}_x$ which is satisfied in our setting.

Property (iii), the dispersive estimate \eqref{eq: UL local decay}, is proved as follows. It suffices to consider the case $s=0$. By interpolation,
\begin{align}
\sup_t\|V_L(t)\|_{L^2_x}
\leq K_0^{1/2}\delta_0^{1/2},
\qquad
\sup_t\|V_L(t)\|_{L^{3/2}_x}
\leq K_0^{2/3}\delta_0^{1/3}.
\label{eq:UL-potential-interpolation}
\end{align}
For every decomposition $h=h_2+h_\infty$, with $h_2\in L^2_x$ and $h_\infty\in L^\infty_x$, we have
\begin{align*}
\|V_L(t)h\|_{L^1_x}
&\leq
\|V_L(t)\|_{L^2_x}\|h_2\|_{L^2_x}
+\|V_L(t)\|_{L^1_x}\|h_\infty\|_{L^\infty_x}\\
&\lesssim_{K_0}
\|h_2\|_{L^2_x}+\|h_\infty\|_{L^\infty_x},
\end{align*}
and
\begin{align*}
\|V_L(t)h\|_{L^2_x}
&\leq
\|V_L(t)\|_{L^\infty_x}\|h_2\|_{L^2_x}
+\|V_L(t)\|_{L^2_x}\|h_\infty\|_{L^\infty_x}\\
&\leq
\eta_0
(\|h_2\|_{L^2_x}+\|h_\infty\|_{L^\infty_x}),
\end{align*}
where $\eta_0=\delta_0+K_0^{1/2}\delta_0^{1/2}$. Taking the infimum over all such decompositions gives
\begin{align}
\|V_L(t)h\|_{L^1_x}
&\lesssim_{K_0}\|h\|_{L^2_x+L^\infty_x},
\label{eq:UL-potential-L1-map}\\
\|V_L(t)h\|_{L^2_x}
&\leq\eta_0\|h\|_{L^2_x+L^\infty_x}.
\label{eq:UL-potential-L2-map}
\end{align}
Thus multiplication by $V_L$ is bounded from $L^2_x+L^\infty_x$ to $L^1_x$, while its norm from $L^2_x+L^\infty_x$ to $L^2_x$ tends to zero as $\delta_0\to0$ with $K_0$ fixed.

Fix $R\geq2$ and $T\geq2R$, and set
\begin{align*}
A_T
=
\sup_{|t_0|\leq T}
\langle t_0\rangle^{3/2}
\|U_L(t_0,0)f\|_{L^2_x+L^\infty_x}.
\end{align*}
For $2R\leq t\leq T$, Duhamel's formula gives
\begin{align}
\begin{split}
U_L(t,0)f
= & e^{-it\Delta}f
-i\int_0^R e^{-i(t-r)\Delta}V_L(r)U_L(r,0)f\,dr - i\int_R^{t-R}e^{-i(t-r)\Delta}V_L(r)U_L(r,0)f\,dr\\
&-i\int_{t-R}^t e^{-i(t-r)\Delta}V_L(r)U_L(r,0)f\,dr.
\end{split}
\label{eq:UL-three-interval-decomposition}
\end{align}

For the first integral, \eqref{eq: mass conservation} and \eqref{eq:UL-potential-interpolation} give
\begin{align*}
\|V_L(r)U_L(r,0)f\|_{L^1_x}
\leq \|V_L(r)\|_{L^2_x}
\|U_L(r,0)f\|_{L^2_x} \leq
K_0^{1/2}\delta_0^{1/2}\|f\|_{L^2_x}.
\end{align*}
Since $t-r\geq t/2$ for $0\leq r\leq R$, the free dispersive estimate yields
\begin{align}
\left\|
\int_0^R e^{-i(t-r)\Delta}V_L(r)U_L(r,0)f\,dr
\right\|_{L^\infty_x}
\lesssim
R K_0^{1/2}\delta_0^{1/2}
t^{-3/2}\|f\|_{L^2_x}.
\label{eq:UL-disp-initial}
\end{align}

For the middle integral, \eqref{eq:UL-potential-L1-map} and the definition of $A_T$ give
\begin{align} \label{eq:UL-middle-bound}
\left\|
\int_R^{t-R}e^{-i(t-r)\Delta}V_L(r)U_L(r,0)f\,dr
\right\|_{L^\infty_x} \lesssim_{K_0}
A_T
\int_R^{t-R}
(t-r)^{-3/2}\langle r\rangle^{-3/2}\,dr \lesssim
R^{-1/2}t^{-3/2}A_T,
\end{align}
where the last inequality is obtained by splitting the time integral at $t/2$.

For the final integral, the unitarity of the free flow and \eqref{eq:UL-potential-L2-map} imply
\begin{align}
\begin{split}
\left\|
\int_{t-R}^t e^{-i(t-r)\Delta}V_L(r)U_L(r,0)f\,dr
\right\|_{L^2_x}
\leq \eta_0 A_T
\int_{t-R}^t\langle r\rangle^{-3/2}\,dr
\lesssim
R\eta_0 t^{-3/2}A_T.
\end{split}
\label{eq:UL disp est 2}
\end{align}
We first choose $R$ sufficiently large, depending only on $K_0$, so that $C_{K_0}R^{-1/2}\leq\frac{1}{10}$.
Then we choose $\delta_0$ sufficiently small, depending on $K_0$ and $R$, so that $CR\eta_0\leq\frac{1}{10}$.
Combining the free dispersive estimate with
\eqref{eq:UL-disp-initial}, \eqref{eq:UL-middle-bound}, and
\eqref{eq:UL disp est 2}, we obtain
\begin{align*}
\sup_{2R\leq t\leq T}
t^{3/2}
\|U_L(t,0)f\|_{L^2_x+L^\infty_x}
\leq
C_{K_0,R}\|f\|_{L^1_x\cap L^2_x}
+\frac15 A_T.
\end{align*}
For $t\leq2R$, \eqref{eq: mass conservation} gives
\begin{align*}
\langle t\rangle^{3/2}
\|U_L(t,0)f\|_{L^2_x+L^\infty_x}
\leq
\langle2R\rangle^{3/2}\|f\|_{L^2_x}.
\end{align*}
It follows that
\begin{align*}
A_T
\leq
C_{K_0,R}\|f\|_{L^1_x\cap L^2_x}
+\frac15 A_T.
\end{align*}
After absorbing the last term and sending $T\to\infty$, we obtain
\begin{align*}
\sup_{t}
\langle t\rangle^{3/2}
\|U_L(t,0)f\|_{L^2_x+L^\infty_x}
\lesssim_{K_0}
\|f\|_{L^1_x\cap L^2_x}.
\end{align*}
Applying the same argument to the translated potential
$V_L(x,s+\cdot)$ proves \eqref{eq: UL local decay} for arbitrary $s$.

We next prove property (iv). Interpolation gives
\begin{align}
\sup_t\|\nabla V_L(t)\|_{L^{3/2}_x}
\leq
K_1^{1/2}\delta_1^{1/2}.
\label{eq:UL-gradient-interpolation}
\end{align}
The commutator identity is
\begin{align*}
D_jU_L(t,s)f
=
U_L(t,s)D_jf
-i\int_s^t
U_L(t,r)(D_jV_L(r))U_L(r,s)f\,dr.
\end{align*}
By the dual of \eqref{eq:UL endpt Strichartz},
\eqref{eq:UL-gradient-interpolation}, and
\eqref{eq:UL endpt Strichartz},
\begin{align*}
\sup_t
\left\|
\int_s^t
U_L(t,r)(D_jV_L(r))U_L(r,s)f\,dr
\right\|_{L^2_x}
&\lesssim
\|(D_jV_L)U_L(\cdot,s)f\|_{L^2_tL^{6/5}_x}\\
&\lesssim
\|D_jV_L\|_{L^\infty_tL^{3/2}_x}
\|U_L(\cdot,s)f\|_{L^2_tL^6_x}\\
&\lesssim_{K_1}
\|f\|_{L^2_x}.
\end{align*}
Combining this estimate with \eqref{eq: mass conservation} proves
\eqref{eq:H1 uniform bd}.

Finally, we prove property (v). In addition to
\eqref{eq:UL-potential-interpolation}, interpolation gives
\begin{align*}
\sup_t\|\nabla V_L(t)\|_{L^2_x}
\leq
K_1^{1/4}\delta_1^{3/4}.
\end{align*}
For $h_2\in W^{1,2}_x$, Hölder's inequality and Sobolev embedding give
\begin{align}
\|V_Lh_2\|_{W^{1,1}_x}
\lesssim
(\|V_L\|_{L^2_x}+\|\nabla V_L\|_{L^2_x})
\|h_2\|_{W^{1,2}_x},
\label{eq:UL-W12-to-W11}
\end{align}
and
\begin{align*}
\|V_Lh_2\|_{W^{1,2}_x}
\lesssim
(\|V_L\|_{L^\infty_x}+\|\nabla V_L\|_{L^3_x})
\|h_2\|_{W^{1,2}_x}.
\end{align*}
For $h_\infty\in W^{1,\infty}_x$, we have
\begin{align*}
& \|V_Lh_\infty\|_{W^{1,1}_x}
\lesssim
(\|V_L\|_{L^1_x}+\|\nabla V_L\|_{L^1_x})
\|h_\infty\|_{W^{1,\infty}_x},
\\ & \|V_Lh_\infty\|_{W^{1,2}_x}
\lesssim
(\|V_L\|_{L^2_x}+\|\nabla V_L\|_{L^2_x})
\|h_\infty\|_{W^{1,\infty}_x}.
\end{align*}
It follows that
\begin{align}
\|V_Lh\|_{W^{1,1}_x}
\lesssim_{K_0,K_1}
\|h\|_{W^{1,2}_x+W^{1,\infty}_x},
\label{eq:UL-potential-W11-map}\\
\|V_Lh\|_{W^{1,2}_x}
\lesssim
\eta_1
\|h\|_{W^{1,2}_x+W^{1,\infty}_x},
\label{eq:UL-potential-W12-map}
\end{align}
where $\eta_1 = \delta_0+\delta_1 +K_0^{1/2}\delta_0^{1/2} +K_1^{1/4}\delta_1^{3/4}$.

In particular, the first map is uniformly bounded, while the norm of the second map tends to zero as $\delta_0+\delta_1\to0$ with $K_0,K_1$ fixed.
Fix $R\geq2$, let $T\geq2R$, and set
\begin{align*}
B_T = \sup_{|t_0|\leq T}
\langle t_0\rangle^{3/2}
\|U_L(t_0,0)f\|_{W^{1,2}_x+W^{1,\infty}_x}.
\end{align*}
For the first time interval in
\eqref{eq:UL-three-interval-decomposition}, the free dispersive estimate for one derivative, \eqref{eq:H1 uniform bd}, and
\eqref{eq:UL-W12-to-W11} give
\begin{align}
\left\|
\int_0^R
e^{-i(t-r)\Delta}V_L(r)U_L(r,0)f\,dr
\right\|_{W^{1,\infty}_x}
\lesssim_{K_0,K_1,R}
t^{-3/2}\|f\|_{W^{1,2}_x}.
\label{eq:UL-derivative-initial}
\end{align}
For the middle interval, \eqref{eq:UL-potential-W11-map} gives
\begin{multline}
\left\|
\int_R^{t-R}
e^{-i(t-r)\Delta}V_L(r)U_L(r,0)f\,dr
\right\|_{W^{1,\infty}_x}\lesssim_{K_0,K_1}
B_T
\int_R^{t-R}
(t-r)^{-3/2}\langle r\rangle^{-3/2}\,dr  \\ 
\lesssim_{K_0,K_1}
R^{-1/2}t^{-3/2}B_T.
\label{eq:UL deriv disp est 1}
\end{multline}
For the final interval, \eqref{eq:UL-potential-W12-map} and the unitarity of the free flow on $W^{1,2}_x$ give
\begin{align}
\begin{split}
&\left\|
\int_{t-R}^t
e^{-i(t-r)\Delta}V_L(r)U_L(r,0)f\,dr
\right\|_{W^{1,2}_x} \lesssim
\eta_1 B_T
\int_{t-R}^t
\langle r\rangle^{-3/2}\,dr \lesssim
R\eta_1t^{-3/2}B_T.
\end{split}
\label{eq:UL deriv disp est 2}
\end{align}
Choose $R$ sufficiently large, depending on $K_0,K_1$, so that $C_{K_0,K_1}R^{-1/2}\leq\frac{1}{10}$,
and then choose $\delta_0,\delta_1$ sufficiently small so that $CR\eta_1\leq\frac{1}{10}$. 
The free term and
\eqref{eq:UL-derivative-initial}--\eqref{eq:UL deriv disp est 2} then give
\begin{align*}
\sup_{2R\leq|t|\leq T}
|t|^{3/2}
\|U_L(t,0)f\|_{W^{1,2}_x+W^{1,\infty}_x}
\leq
C_{K_0,K_1,R}
\|f\|_{W^{1,1}_x\cap W^{1,2}_x}
+\frac15B_T.
\end{align*}
For $|t|\leq2R$, \eqref{eq:H1 uniform bd} gives
\begin{align*}
\langle t\rangle^{3/2}
\|U_L(t,0)f\|_{W^{1,2}_x+W^{1,\infty}_x}
\lesssim_{K_0,K_1,R}
\|f\|_{W^{1,2}_x}.
\end{align*}
Therefore
\begin{align*}
B_T
\leq
C_{K_0,K_1,R}
\|f\|_{W^{1,1}_x\cap W^{1,2}_x}
+\frac15B_T.
\end{align*}
Absorbing the last term, sending $T\to\infty$, and translating the initial time proves
\eqref{eq: UL local decay deriv}.


\end{proof}


Next we establish an improved dispersive estimate when the propagator $U_L(t,s)$ is sandwiched between two functions supported in disjoint spacetime cones. In preparation for this result we use the following `cancellation lemma', essentially from \cite[Section~2]{JSS-Schrodinger-decay}:

\begin{lemma}[Cancellation Lemma]\label{lem:cancellation} Let $V$ be a multiplication operator with Fourier transform in $L^1$. Then we have an estimate on the $L^p \to L^p$ operator norm for all $p \in [1, \infty]$:
\begin{equation}\label{eq:cancellation}
    \big\| e^{-i t \Delta} V e^{it\Delta} \big\|_{L^p(\R^3) \to L^p(\R^3)} \leq (2\pi)^n \| \hat V \|_{L^1(\R^3)}. 
\end{equation}
\end{lemma}

\begin{proof} 
This follows by writing $V$ as a superposition of plane waves, via the Fourier transform. Then for each plane wave component $e^{i x \cdot k}$, by direct calculation, we find that  
$e^{-i t \Delta} e^{ix \cdot k} e^{it\Delta}$ is a product of a translation operator, modulation operator and multiplication by a complex unit, so has norm $1$ on every $L^p$ space. The general result follows by applying Minkowski's integral inequality. 
\end{proof}

\begin{lemma}
\label{lemma: UL-non-stationary} Assume that $V_L$ satisfies both conditions of Lemma~\ref{lem:UL properties}. Let $\chi_1$ and $\chi_2$ be conic cutoffs, i.e. smooth functions homogeneous of degree zero in $(x, t)$, with disjoint supports. Moreover, we assume that either (i) the support of $\chi_1$ is a cone with compact cross-section, or (ii) $1 - \chi_1$ has support in a cone with compact cross-section. 
Then we have an estimate for $t \geq s \geq 1$:
\begin{equation}\label{eq: UL-non-stationary}
    \big\| \chi_1(\cdot/t) U_L(t,s) \chi_2(\cdot/s) \big\|_{W^{1,1}\cap H^1\to L^2+L^\infty}\lesssim \frac{1}{\langle t-s\rangle^{1/2}} \times \frac{1}{t}. 
\end{equation}
    
\end{lemma}

\begin{proof}
We first prove in case (i), where the support of $\chi_1$ is a cone with compact cross-section. 
For this, we compute $\chi_1 U_L(t,s) \chi_2$ by moving $\chi_1$ through $U_L(t,s)$. 
Fix positive $\delta \ll 1$ such that when $(1-\delta) t \leq s \leq t$, we have
\begin{equation}\label{eq:chii disjoint support}
    \supp \chi_1(\cdot/t) \cap \supp \chi_2(\cdot/s) = \emptyset.
\end{equation}
This is possible since the supports lie in two disjoint closed cones, at least one of which has compact cross-section. So the two cross-sections are a positive distance apart (in terms of the parameter $x/2t$), which implies \eqref{eq:chii disjoint support}. 

For $s \in [0,(1-\delta)t]$, we have $|t-s| \sim t$ and we apply \eqref{eq: UL local decay} to obtain the desired estimate.
Now consider $s \in [(1-\delta)t,t]$. Using the disjoint support property of $\chi_1$ and $\chi_2$,
we have 
\begin{equation}\label{eq: FM2 move}
   \begin{aligned}
 \chi_1(\cdot,t) U_L(t,s) \chi_2(\cdot, s) =& \int_s^t \frac{d}{ds'} \Big( U_L(t, s') \chi_1(\cdot, t) U_L(s',s) \chi_2(\cdot, s) \Big) \, ds' \\
       &=\int_s^t U_{L}(t,s') i[\Delta + V_L, \chi_1(\cdot, t)] U_L(s',s) \chi_2(\cdot, s) ds'\\
       =& \int_s^t U_{L}(t,s') i[\Delta, \chi_1(\cdot, t)] U_L(s',s) \chi_2(\cdot, s) ds'. 
   \end{aligned} 
\end{equation}

Using 
\begin{equation}
\begin{aligned}
    [\Delta, \chi_1(x, t)]=&-2\sum\limits_j[\partial_{x_j}\chi_1(x, t)]\partial_{x_j}+\Delta\chi_1(x, t)\\
    = & -2\sum\limits_j[\partial_{x_j}\chi_1(x, t)]\partial_{x_j} +O(t^{-2})
    \end{aligned}
\end{equation}
in~\eqref{eq: FM2 move} yields 
\begin{equation}\label{eq: Fmove U1}
\begin{aligned}
    & \chi_1(x, t) U_L(t,s) \chi_2(y,s)\\
    = & -2i\sum\limits_{j}\int_s^t U_L(t,s')\partial_{x_j}[\chi_1(x, t)]\partial_{x_j}U_L(s',s)\chi_2(y,s) ds'+O(t^{-2}). 
\end{aligned}
\end{equation}
The zeroth order, $O(t^{-2})$ term can be treated in the same way as the leading order term below; we only treat the leading part in detail.

We now express $U_L$ in terms of the free propagator.  By Duhamel's formula we have
\begin{align}
U_L(s',s)
&=e^{-i(s'-s)\Delta}
-i\int_s^{s'} e^{-i(s'-r)\Delta}V_L(r)U_L(r,s)dr,
\label{eq:UL-expansion-1}\\
U_L(t,s') &= e^{-i(t-s') \Delta} - i \int_{s'}^t U_L(t, r') V_L(r') e^{-i(r'-s')\Delta} \, dr'  \label{eq:UL-expansion-2}
\end{align}

After inserting \eqref{eq:UL-expansion-1} and \eqref{eq:UL-expansion-2} into \eqref{eq: Fmove U1}, the leading term is 
\begin{align}
\begin{split}
&\chi_1(x,t)U_L(t,s)\chi_2(y,s) \\
={}& -2i\sum_j \int_s^t e^{-i(t-s')\Delta}[\partial_{x_j}\chi_1(x,t)]\partial_{x_j}e^{-i(s'-s)\Delta}\chi_2(y,s)\,ds' 
\\ &-2i\sum_j \int_s^t e^{-i(t-s')\Delta}[\partial_{x_j}\chi_1(x,t)]\partial_{x_j}\Big(-i\int_s^{s'} e^{-i(s'-r)\Delta}V_L(r)U_L(r,s)\,dr\Big)\chi_2(y,s)\,ds' \\
&-2i\sum_j \int_s^t \Big(-i\int_{s'}^t U_L(t,r')V_L(r')e^{-i(r'-s')\Delta}\,dr'\Big)[\partial_{x_j}\chi_1(x,t)]\partial_{x_j}e^{-i(s'-s)\Delta}\chi_2(y,s)\,ds' \\
&+2i\sum_j \int_s^t \Big(\int_{s'}^t U_L(t,r')V_L(r')e^{-i(r'-s')\Delta}\,dr'\Big)[\partial_{x_j}\chi_1(x,t)] \times  \\
 &\phantom{aaaaaaaaaaaaaaaaaaaa}\times \partial_{x_j}\Big(\int_s^{s'} e^{-i(s'-r)\Delta}V_L(r)U_L(r,s)\,dr\Big) \chi_2(y,s)\,ds'
\\& :=T_{FF}+T_{FD}+T_{DF}+T_{DD}. 
\end{split}
\label{eq:expanded}
\end{align}

\textbf{The term $T_{FF}$.} 
For the free term $T_{FF}$, use that $\partial_{x_j}$ commutes with the free propagator:
\begin{align}
&-2i\int_s^t e^{-i(t-u)\Delta}
[\partial_{x_j} \chi_1(x, t)]
\partial_{x_j}e^{-i(u-s)\Delta}
\chi_2(y,s)fdu \nonumber\\
&=
-2i\int_s^t
\Big(e^{-i(t-u)\Delta}
[\partial_{x_j}\chi_1(x, t)]
e^{i(t-u)\Delta} \Big)
e^{-i(t-s)\Delta}
\partial_{x_j}(\chi_2(y,s)f)du .
\label{eq:free-conjugated}
\end{align}

If $t-s\ge 1$, then \eqref{eq:free-conjugated}, Lemma~\ref{lem:cancellation} (noting that $\| \widehat{\partial_{x_j} \chi_1}(\xi,t)\|_{L^1} < \infty$ using the compact support of $\chi_1(\cdot, t)$, and is $O(t^{-1})$ by scaling),
and the free dispersive estimate give 
\begin{align}
&\Big\| \int_s^t
(e^{-i(t-u)\Delta}
\partial_{x_j}(\chi_1(x, t))
e^{i(t-u)\Delta})
e^{-i(t-s)\Delta}
\partial_{y_j}(\chi_2(y, s) f(y))du
\Big\|_{L^\infty_y}
\nonumber\\
&\lesssim
\int_s^t
\frac{1}{t}
(t-s)^{-3/2}
\|\partial_{y_j}(\chi_2(y, s) f(y))\|_{L^1_y}du
\nonumber\\
&\lesssim
\frac{1}{t(t-s)^{1/2}}
\|f\|_{W^{1,1}_x}.
\label{eq:free-first-order-large-time}
\end{align}
If $0<t-s\le 1$, then unitarity on $L^2_x$ gives
\begin{align}
&\Big\| \int_s^t
(e^{-i(t-u)\Delta}
\partial_{x_j}(\chi_1(x, t))
e^{i(t-u)\Delta})
e^{-i(t-s)\Delta}
\partial_{y_j}(\chi_2(y, s)f(y))du
\Big\|_{L^2_x}
\nonumber\\
&\lesssim
\int_s^t
\frac{1}{t}
\|\partial_{y_j}(\chi_2(y,s) f(y))\|_{L^2_y}du
\nonumber\\
&\lesssim
\frac{1}{t}
\|f\|_{H^1_x}.
\label{eq:free-first-order-short-time}
\end{align}
Therefore the first-order part satisfies
\begin{align}
&\| \sum_j\int_s^t e^{-i(t-u)\Delta}
\partial_{x_j}(\chi_1(x, t))
\partial_{x_j}e^{-i(u-s)\Delta}
\chi_2(\cdot, s)du
\|_{W^{1,1}_x\cap H^1_x\to L^2_x+L^\infty_x}
\nonumber\\
&\lesssim
\frac{1}{t\langle t-s\rangle^{1/2}} .
\label{eq:free-first-order-final}
\end{align}

The zeroth-order part of the commutator is treated in the same way and satisfies the same estimate (and is even better, as there is an additional decaying factor of $t^{-1}$); we omit the details. In summary, the $T_{FF}$ term satisfies the required estimate.

\textbf{The term $T_{FD}$.} 
Rewriting $T_{FD}$ as
\begin{align}
\begin{split}
T_{FD} &=-2\sum_j\int_s^t\!\!\int_s^{s'}\big(e^{-i(t-s')\Delta}[\partial_{x_j}\chi_1(x,t)]e^{i(t-s')\Delta}\big)
\\ &\qquad\times e^{-i(t-r)\Delta}\partial_{x_j}\big(V_L(r)U_L(r,s)\chi_2(y,s)f\big)\,dr\,ds'.
\end{split}
\end{align}
By \eqref{eq: UL local decay deriv}, $U_L(r,s)\chi_2(y,s)f\in W^{1,2}_x+W^{1,\infty}_x$ with norm $\lesssim\langle r-s\rangle^{-3/2}\|f\|_{W^{1,1}_x\cap H^1_x}$. By Lemma~\ref{lem:UL properties} (v), applying $V_L$ then $\partial_{x_j}$ lands in $L^1_x\cap L^2_x$ with norm $\lesssim\langle r-s\rangle^{-3/2}\|f\|_{W^{1,1}_x\cap H^1_x}$. 
Using Lemma~\ref{lem:cancellation} and the free dispersive estimate, the integrand is bounded in $L^2_x+L^\infty_x$ by $\tfrac{1}{t}\langle t-r\rangle^{-3/2}\langle r-s\rangle^{-3/2}\|f\|_{W^{1,1}_x\cap H^1_x}$, uniformly in $s'$. 
Integrating $s'\in[r,t]$ gives a factor $t-r$, and we have
\begin{align}
\| T_{FD}\|_{W^{1,1}_x\cap H^1_x\to L^2_x+L^\infty_x}
\lesssim \frac{1}{t}\int_s^t(t-r)\langle t-r\rangle^{-3/2}\langle r-s\rangle^{-3/2}\,dr
\lesssim \frac{1}{t\,\langle t-s\rangle^{1/2}}.
\end{align}

\textbf{The term $T_{DF}$.}
Rewriting $T_{DF}$ as
\begin{align}
\begin{split}
T_{DF}
&=-2\sum_j\int_s^t\!\!\int_{s'}^t U_L(t,r')V_L(r')\big(e^{-i(r'-s')\Delta}\partial_{x_j}(\chi_1(x, t))e^{i(r'-s')\Delta}\big)
\\ &\qquad\times e^{-i(r'-s)\Delta}\partial_{x_j}\big(\chi_2(y,s)f\big)\,dr'\,ds'.
\end{split}
\end{align}
Now $\partial_{x_j}(\chi_2(y,s)f)\in L^1_x\cap L^2_x$ with norm $\lesssim\|f\|_{W^{1,1}_x\cap H^1_x}$ directly. 
Using Lemma~\ref{lem:cancellation} and the free dispersive estimate places the resulting function in $L^2_x+L^\infty_x$ with norm $\lesssim\tfrac1t\langle r'-s\rangle^{-3/2}\|f\|_{W^{1,1}_x\cap H^1_x}$, after which multiplication by $V_L$ maps $L^2_x+L^\infty_x$ to $L^1_x\cap L^2_x$ and $U_L(t,r')$ via \eqref{eq: UL local decay} contributes an operator norm factor of $\langle t-r'\rangle^{-3/2}$. The integrand bound $\tfrac{1}{t}\langle t-r'\rangle^{-3/2}\langle r'-s\rangle^{-3/2}$ is independent of $s'$; integrating $s'\in[s,r']$ gives a factor $r'-s$ and we have
\begin{align}
\| T_{DF}\|_{W^{1,1}_x\cap H^1_x\to L^2_x+L^\infty_x}
\lesssim \frac{1}{t}\int_s^t(r'-s)\langle t-r'\rangle^{-3/2}\langle r'-s\rangle^{-3/2}\,dr'
\lesssim \frac{1}{t\,\langle t-s\rangle^{1/2}}.
\end{align}

\textbf{The term $T_{DD}$.} Conjugating $\partial_{x_j}(\chi_1(x, t))$ by $e^{-i(r'-s')\Delta}$ and combining the two free propagators,
\begin{align}
\begin{split}
T_{DD}
&= 2i\sum_j\int_s^t\int_{s'}^t \int_s^{s'} U_L(t,r')V_L(r')\big(e^{-i(r'-s')\Delta}[\partial_{x_j}\chi_1(x, t)]e^{i(r'-s')\Delta}\big)
\\ &\qquad\times e^{-i(r'-r)\Delta}\partial_{x_j}\big(V_L(r)U_L(r,s)\chi_2(y,s)f\big)\,dr\,dr'\,ds'.
\end{split}
\end{align}
By the discussion for $T_{FD}$, we have $\partial_{x_j}(V_L(r)U_L(r,s)\chi_2(y,s)f)\in L^1_x\cap L^2_x$ with norm $\lesssim\langle r-s\rangle^{-3/2}\|f\|_{W^{1,1}_x\cap H^1_x}$ by \eqref{eq: UL local decay deriv}. For the middle factor, we apply Lemma~\ref{lem:cancellation} together with the free dispersive estimate, which shows that it is bounded by $\frac{1}{t}\langle r'-r\rangle^{-3/2}$, and with the outer block of $T_{DF}$, namely multiplication by $V_L$ followed by $U_L(t,r')$ via \eqref{eq: UL local decay} contributing $\langle t-r'\rangle^{-3/2}$, the integrand is bounded in $L^2_x+L^\infty_x$ by
\begin{align}
\frac{1}{t}\,\langle t-r'\rangle^{-3/2}\langle r'-r\rangle^{-3/2}\langle r-s\rangle^{-3/2}\|f\|_{W^{1,1}_x\cap H^1_x}
\end{align}
uniformly in $s'$. Integrating $s'\in[r,r']$ gives a factor $r'-r$, and we have
\begin{align}
\begin{split}
\| T_{DD}\|_{W^{1,1}_x\cap H^1_x\to L^2_x+L^\infty_x}
&\lesssim \frac{1}{t}\int_s^t\langle r-s\rangle^{-3/2}\Big(\int_r^t(r'-r)\langle t-r'\rangle^{-3/2}\langle r'-r\rangle^{-3/2}\,dr'\Big)dr
\\ &\lesssim \frac{1}{t}\int_s^t\langle r-s\rangle^{-3/2}\langle t-r\rangle^{-1/2}\,dr
\lesssim \frac{1}{t\,\langle t-s\rangle^{1/2}}.
\end{split}
\end{align}

This proves the required estimate and completes the proof under case (i). The proof in case (ii) only requires two minor adjustments. First, condition \eqref{eq:chii disjoint support}  still holds for some $\delta > 0$, as when $\chi_1(x/t) = 1$ for all large $|x/t|$, then necessarily $\chi_2$ is supported in a cone with compact cross-section, so the two cross-sections are still a positive distance apart. 
Second, we note that Lemma~\ref{lem:cancellation} still applies in case (ii), since $\partial_{x_j}\chi_1(\cdot,t)$ and $\Delta\chi_1(\cdot,t)$ have compact spatial support, and hence their Fourier transforms are in $L^1$, with the same bounds as in case (i). The rest of the proof proceeds unchanged.
\end{proof}


\section{Strichartz estimate in a cone} \label{sec:Strichartz}

We prove the endpoint Strichartz estimate localized to a spacetime cone in this section.

\begin{proposition}\label{prop:u exterior decay}
Assume that the nonlinear potential $\mc{N}$ satisfies conditions \eqref{eq:Ncondition 1} -- \eqref{eq:Ncondition 3}. Assume that $\| u \|_{L^\infty(C^\circ)}$ is sufficiently small on some open spacetime cone $C^\circ$. Let $\chi_1$ be a conic cutoff supported in $C^\circ$, such that either (i) $\chi_1$ has compact cross-sectional support (see Section~\ref{sec:spacetime cones} for the definition) or (ii) $1-\chi_1$ has compact cross-sectional support. 
Let $u$ be a solution to \eqref{eq:nls} with $H^2(\R^3)$ initial condition. Then it satisfies the (long-time) Strichartz estimate
\begin{equation}\label{eq:input-L2-2}
 \norm[L^2_t L^{6/5}_x]{\chi_1(\cdot, t) \mathcal{N}(|u(t)|)\,u(t)\; } 
 \;\lesssim\;1. 
\end{equation}

\end{proposition}

\begin{proof}
We first cut off the nonlinear potential to a spacetime cone so that it has small norm in $L^\infty_t L^{3/2}_x$, and so that its gradient is small in $L^\infty_t L^3_x$, using \eqref{eq:nonlinear pot cones} and \eqref{eq:gradient nonlinear pot cones}. This will put us in a position to apply Lemma~\ref{lem:UL properties}. We do this by choosing a conic cutoff $\chi_1'$ equal to $1$ on a neighbourhood of the support of $\chi_1$, with support contained in $C^\circ$. Then we define $\chi_2 = 1 - \chi_1'$, which therefore has disjoint support from $\chi_1$. Our potential $V_L$ is defined by $V_L := \mc{N}(|u|) \chi_1'$. By construction, since $|u|$ is pointwise small on the support of $\chi_1'$, $V_L$ is small in $L^\infty_t L^{3/2}_x$. Similarly, $\nabla V_L$ will be small in $L^\infty_t L^3_x$, as it also contains at least one factor of $|u|$ (recall that we can always pull out a factor $|u|^2$ from $\mc{N}$, thanks to \eqref{eq:Ncondition 1}). We define $U_L(t, s)$ to be the propagator for the time-dependent Schr\"odinger operator $\Delta + V_L$. Then all the results of Lemmas~\ref{lem:UL properties} and \ref{lemma: UL-non-stationary} are available. 

We express $u(t)$ in terms of the $U_L$ evolution and a Duhamel term:
\begin{equation}
u(t) = U_L(t, 0) u(0) - i \int_0^t U_L(t, s) \chi_2(\cdot, s) \mc{N}(|u(s)|) u(s) \, ds. 
\end{equation}
Thus 
\begin{equation}\label{eq: u exterior expression}
\chi_1(\cdot, t) u(t) = \chi_1(\cdot, t) U_L(t, 0) u(0) - i \chi_1(\cdot, t) \int_0^t U_L(t, s) \chi_2(\cdot, s) \mc{N}(|u(s)|) u(s) \, ds. 
\end{equation}

    The first term is in $L^2_t L^6_x$,  thanks to \eqref{eq:UL endpt Strichartz}. Now we consider the integral term. 
    We decompose the integral into the part with $s \leq \frac{t}{2}$ and $s \in [\frac{t}{2},t]$.
For the first part, we use the dispersive estimate \eqref{eq: UL local decay} and $|t-s| \gtrsim t$ to conclude that for fixed $s$, the $L^2 + L^\infty$-norm is bounded by $t^{-1/2}$.


For the part with $s \in [\frac{t}{2},t]$, we use \eqref{eq:nonlinear Lp bound} for $p=1$ and $p=2$ (see Remark~\ref{rem:ext} for the $p=2$ case), apply  Lemma~\ref{lemma: UL-non-stationary}  and integrate in $s$ to see that this part is bounded by 
\begin{equation}
  \| \chi_1(\cdot, t) \int_{\frac{t}{2}}^t U_L(t, s) \chi_2(\cdot, s) \mc{N}(|u(s)|) u(s) \, ds  \|_{L^2+L^\infty} \lesssim t^{-1/2}.
\end{equation}

In sum, we have proved that  $\chi_1(\cdot, t) u(t)$ is the sum of a term $u'$ in $L^2_tL^6_x$, plus a term bounded, pointwise, in $t$, by  $C t^{-1/2}$ in the space $L^2+L^\infty$. Now to treat $\chi_1(\cdot, t) \mc{N}(|u(t)|) u(t)$, we factor out $\chi_1 |u|^2 u$, and the remaining factor is bounded thanks to Proposition~\ref{prop:Linfty} and \eqref{eq:Ncondition 1}. So it is sufficient to treat the case that $\mc{N}(|u|) = |u|^2$. 
Let 
\begin{align*}
    \chi_1u = u' + u_{2} + u_{\infty}
\end{align*}
be a division of $u$ into $u' \in L^2_t L^6_x$ as above, an $O(t^{-1/2})$ part $u_{2}$ in $L^2$ and an $O(t^{-1/2})$ part $u_{\infty}$ in $L^\infty$. 
Then we notice that, pointwise, there is always at least one of $|u'|,|u_2|,|u_\infty|$ that is comparable to $|\chi u|$, so
\begin{align*}
|\chi_1 u|^3 \lesssim |\chi_1u||u'|^2+
|\chi_1u|^{4/3}|u_2|^{5/3}+|\chi_1u|^{5/3}|u_\infty|^{4/3}. 
\end{align*}
Taking the $L^{6/5}$-norm, we have
\begin{align*}
\||\chi_1 u|^3\|_{L^{6/5}} \lesssim & \||\chi_1u||u'|^2\|_{L^{6/5}}+
\||\chi_1u|^{4/3}|u_2|^{5/3}\|_{L^{6/5}}+
\||\chi_1u|^{5/3}|u_\infty|^{4/3}\|_{L^{6/5}}
\\ \lesssim &
\|\chi_1 u\|_{L^3}^2\|u'\|_{L^6}
+ \|\chi_1u\|_{L^\infty}^{4/3}\|u_2\|_{L^2}^{5/3}
+ \|\chi_1u\|_{L^2}^{5/3}\|u_\infty\|_{L^\infty}^{4/3}
\lesssim \|u'\|_{L^6}+t^{-5/6}+t^{-2/3}.
\end{align*}
In the second line, we used the fact that $\chi_1 u$ is uniformly bounded in $L^p$ for $2 \leq p \leq \infty$, due to mass conservation and Proposition~\ref{prop:Linfty}. 
Since $u' \in L^2_tL^6_x$, this
is in $L^2_t$. This proves \eqref{eq:input-L2-2} with $\chi_1$ replaced with $\chi_1^3$. To obtain the factor $\chi_1$ we rerun the argument with a function $\chi_1''$ that is equal to $1$ on the support of $\chi_1$ and supported where $\chi_1' = 1$, and then note that $\chi_1 = (\chi_1'')^3 f$ for some bounded function $f$. 
\end{proof}

\begin{remark}
This Proposition provides one of the tools we need in the following sections to prove scattering. Recall that by Duhamel's formula, we can write the solution $u(t)$ to \eqref{eq:nls} as 
$$
u(t) = e^{-it\Delta} u(0) + (-i) \int_0^t e^{-i(t-s)\Delta} \mc{N}(|u(s)|) u(s) \, ds. 
$$
Suppose the assumptions of Proposition~\ref{prop:u exterior decay} are satisfied. Then we write the integral as 
$$
e^{-it\Delta} \Big( \int_0^t e^{is\Delta} \chi_1(\cdot, s) \mc{N}(|u(s)|) u(s) \, ds + \int_0^t e^{is\Delta} (1-\chi_1)(\cdot, s) \mc{N}(|u(s)|) u(s) \, ds \Big).
$$
Then Proposition~\ref{prop:u exterior decay} provides a means to show that the first term scatters. Indeed, the convergence of the integral 
$$
\int_0^t e^{is\Delta} \chi_1(\cdot, s) \mc{N}(|u(s)|) u(s) \, ds
$$
as $t \to \infty$ follows from \eqref{eq:input-L2-2} and  the adjoint endpoint Strichartz estimate, since the $L^2_t$-norm of $\| \chi_1(s) \mc{N}(|u(s)|)u(s) \|_{L^{6/5}_x}$ on $[t, \infty)$ tends to zero as $t \to \infty$. The second term, with factor $1-\chi_1$, will be analyzed using non-stationary phase estimates. 
\end{remark}

\section{Scattering outside a single cone} \label{sec:onebubble}
\label{sec:one-bubble}


We prove the special case of Theorem~\ref{thm:main} in which the cone $C^\circ$ takes the form $\{ |x| > c_0|t|\}$ for a fixed $c_0 \in \R$ in this section, which may be large or small. We call this the `one-bubble setting'. (In this one-bubble setting, there is no need for the complementary cone, where $|u|$ is not necessarily small,  to be `thin', as formulated in Theorem~\ref{thm:main}.)  There is no loss of generality in assuming that the cone is centred at $x=0$, that is, at zero velocity, as this can always be arranged with a Galilean transformation. 
This special case illustrates the proof strategy and the proof can be adapted to the more complicated multi-bubble setting.

We fix $\delta_0 \ll 1$, and employ a smooth non-decreasing function $F:\R \to[0,1]$ such that
\begin{align}
\begin{split}
F(\lambda)=0 \quad (\lambda\leq 1-\delta_0), \qquad F(\lambda)=1 \quad (\lambda\geq1), \qquad \sqrt{F'} \in C^\infty(\R). 
\end{split} 
\label{eq:F-convention}
\end{align}
And we denote
\begin{align}\label{eq:Fscaled}
    \overline F=1-F, \quad F_A(\lambda)=F(\lambda/A), \quad \overline F_A=1-F_A.
\end{align}

Our task is to prove scattering in the region $|x| \geq c t$ for any $c > c_0$. 
We choose intermediate constants $c_0 < c_1 < c_2 < c_3  < c_4 = c$ and assume that $\delta_0$ is small compared $c_{j+1}/c_j - 1$, for $j = 0 \dots 3$. 
We recall that under the hypothesis \eqref{eq:cone condition 2} in the one bubble zero-velocity setting, Proposition~\ref{prop:u exterior decay} and its proof give
\begin{align}
\|F(|x|/(c_1t))\mc{N}(|u|)u\|_{L^2_tL^{6/5}_x([T,\infty)\times\R^3)}
, \quad \|F(|x|/c_1t)|u|^2u\|_{L^2_tL^{6/5}_x([T,\infty)\times\R^3)}<\infty. \label{eq:one-bubble-exterior-source}
\end{align}
In this setting, our goal is to prove that there exists $u_+\in H^1(\R^3)$ such that
\begin{align}
\|F(|x|/ct)(u(t)-e^{-it\Delta}u_+)\|_{L^2_x}\longrightarrow0. \label{eq:goal-one-bubble}
\end{align}

\subsection{Weak limit}


\begin{lemma}
\label{lem:weak-limit-u-plus}
There exists $u_+\in H^1(\R^3)$ such that
\begin{align}
e^{it\Delta}u(t)\rightharpoonup u_+ \quad\text{weakly in }L^2_x, \qquad \|u_+\|_{H^1_x}\leq\liminf_{t\to\infty}\|u(t)\|_{H^1_x}. \label{eq:weak-limit}
\end{align}
\end{lemma}

\begin{proof}
The existence of such weak limits is well known; see for example \cite{Nakanishi-asymptotic-free,Strauss-book}. For completeness we provide a proof. 
Set $w(T):=e^{iT\Delta}u(T)$, and part (iii) of Proposition~\ref{prop:N basic properties} shows that 
\begin{align}\label{eq:source-L1-standing}
\sup_{s\ge 0}\|\mathcal N(|u(s)|)u(s)\|_{L^1_x} =  C(\mathscr{M},\mathscr{E}) < \infty.
\end{align}

By mass conservation, $\|w(T)\|_{L^2_x}=\|u(T)\|_{L^2_x}=\mathscr{M}^{1/2}$ for all $T$, so the family $(w(T))_{T\ge 0}$ is bounded in $L^2_x$. Applying $e^{iT\Delta}$ to the Duhamel formula and using $e^{iT\Delta}e^{-i(T-s)\Delta}=e^{is\Delta}$, we obtain
\begin{align}\label{eq:w-duhamel}
w(T) = u(0) - i\int_0^T e^{is\Delta}\,\mathcal N(|u(s)|)u(s)\,ds.
\end{align}
Consequently, for $T_2>T_1$ and any $\varphi\in L^2_x$, using $(e^{is\Delta})^*=e^{-is\Delta}$, we have
\begin{align}\label{eq:w-pairing-diff}
\langle w(T_2)-w(T_1),\varphi\rangle = -i\int_{T_1}^{T_2}\langle \mathcal N(|u(s)|)u(s),\,e^{-is\Delta}\varphi\rangle\,ds.
\end{align}

Let $\varphi\in L^1_x\cap L^2_x$. By the free dispersive estimate $\|e^{-is\Delta}\varphi\|_{L^\infty_x}\le C s^{-3/2}\|\varphi\|_{L^1_x}$ together with \eqref{eq:source-L1-standing}, for $T_2>T_1\ge 1$ the pairing \eqref{eq:w-pairing-diff} is estimated by
\begin{align}\label{eq:w-cauchy-rate}
|\langle w(T_2)-w(T_1),\varphi\rangle| \le \int_{T_1}^{T_2}\|\mathcal N(|u(s)|)u(s)\|_{L^1_x}\|e^{-is\Delta}\varphi\|_{L^\infty_x}\,ds \le 2C(\mathscr{M},\mathscr{E})\,\|\varphi\|_{L^1_x}\,T_1^{-1/2}.
\end{align}
Hence $T\mapsto\langle w(T),\varphi\rangle$ is Cauchy as $T\to\infty$ for every $\varphi\in L^1_x\cap L^2_x$.
In addition, $|\la w(T), \varphi \ra| \leq \mathscr{M}^{1/2} \| \varphi \|_{L^2}$ uniformly.
Then a standard density argument extends to the general case with $\varphi \in L^2_x$. Thus the linear functional $l(\varphi) := \lim_{T \to \infty} \la w(T), \varphi \ra$ is bounded.  
By the Riesz representation theorem, we know there exists $u_+ \in L^2_x$ such that 
$$\la u_+,\varphi \ra = \lim_{T\to\infty}\la w(T),\varphi \ra$$ 
holds for all $\varphi \in L^2$. Thus $u_+$ is the weak limit of $w(T)$ as $T \to \infty$, and it satisfies $\| u_+ \|_{L^2_x} \leq \lim\inf_{T \to \infty} \| w(T) \|_{L^2_x} = \mathscr{M}^{1/2}$. 

Next we show that $u_+ \in H^1$. We recall from Proposition~\ref{prop:N basic properties} that $u(t)$ is uniformly bounded in $H^1$. 
Choose a sequence of time $t_n \to \infty$ such that
\begin{align*}
    \|u(t_n)\|_{H^1}\to \liminf_{t\to\infty}\|u(t)\|_{H^1}.
\end{align*}
After passing to a subsequence
$$e^{it_n\Delta}u(t_n)\rightharpoonup v
\quad\text{weakly in }H^1.$$
This convergence is stronger than weak convergence in $L^2$, so such a subsequence also converges weakly in $L^2$. As the whole sequence converges weakly in $L^2$ to $u_+$, it follows that $v = u_+$. Thus $u_+$ is in $H^1$ with the claimed norm bound. 
\end{proof}

\subsection{Scattering with longitudinal cut-off}
Fix $e\in S^2$ and write
\begin{align}
x^\parallel=x\cdot e, \qquad D^\parallel=D\cdot e. \label{eq:one-bubble-parallel-variables}
\end{align}
We will call $x^\parallel$ and $\xi^\parallel$ \emph{longitudinal} spatial and frequency variables, respectively, and $D^\parallel$ the longitudinal derivative. 
Define longitudinal conic cutoffs $\Xi_\pm$, depending only on $x^\parallel$ and $t$:
\begin{align}
\Xi_\pm(t)=F\big(\pm \frac{x^\parallel}{c_3 t}\big), 
\label{eq:3d-outer-regions}
\end{align}
and frequency cutoffs, which we call resonant and non-resonant:
\begin{align}
Q_{\res, \pm}=F\big(\pm \frac{2D^\parallel}{c_2}\big), \qquad Q_{\nr, \pm}=1-Q_{\res, \pm}. 
\label{eq:3d-Pi-split}
\end{align}
\begin{proposition}
\label{prop:one-bubble-directional-scattering}
For every $e\in S^2$,
\begin{align}
\|\Xi_-(t)(u(t)-e^{-it\Delta}u_+)\|_{L^2_x}\longrightarrow0. \label{eq:one-bubble-directional-scattering}
\end{align}
Indeed, the following statements hold, where we employ a cutoff $F(|D^\parallel|/\sqrt{t})$, to restrict to frequencies with $|\xi^\parallel| \gtrsim \sqrt{t}$, as well as its complementary cutoff $\overline{F}(|D^\parallel|/\sqrt{t})$. They 
 collectively imply \eqref{eq:one-bubble-directional-scattering}. 
\begin{align}
  \|F(\frac{|D^\parallel|}{\sqrt t})u(t)\|_{L^2_x}+\|F(\frac{|D^\parallel|}{\sqrt t})e^{-it\Delta}u_+\|_{L^2_x}&\longrightarrow 0. \label{eq:IIb-vanish-1} \\
 \|\Xi_-(t)Q_{\res, -}\overline{F}(\frac{|D^\parallel|}{\sqrt t})(u(t)-e^{-it\Delta}u_+)\|_{L^2_x}&\longrightarrow0. \label{eq:one-bubble-resonant-scattering-1} \\
 \|\Xi_-(t)Q_{\nr, -}\overline{F}(\frac{|D^\parallel|}{\sqrt t}) u(t)\|_{L^2_x} + \|\Xi_-(t)Q_{\nr, -}\overline{F}(\frac{|D^\parallel|}{\sqrt t})e^{-it\Delta} u_+ \|_{L^2_x}&\longrightarrow0. \label{eq:3d-outer-mid-vanishes-1}
\end{align}
Moreover, the corresponding statements hold for $\Xi_+$. 
\end{proposition}

\begin{proof}
It suffices to prove the statement for $\Xi_-$; the proof for $\Xi_+$ is obtained by replacing $x^\parallel$ and $D^\parallel$ by $-x^\parallel$ and $-D^\parallel$. It is also straightforward to show that \eqref{eq:IIb-vanish-1} --- \eqref{eq:3d-outer-mid-vanishes-1} imply \eqref{eq:one-bubble-directional-scattering}, so we focus on proving these three limiting statements. 

The first statement, \eqref{eq:IIb-vanish-1}, is straightforward. 
Indeed, the uniform $H^1_x$ bound on $u(t)$ and Lemma~\ref{lem:weak-limit-u-plus} immediately imply the stronger statement 
\begin{align}
\|F(\frac{|D^\parallel|}{\sqrt t})u(t)\|_{L^2_x}+\|F(\frac{|D^\parallel|}{\sqrt t})e^{-it\Delta}u_+\|_{L^2_x}\lesssim t^{-1/2}. \label{eq:IIb-vanish}
\end{align}

To prove the second statement \eqref{eq:one-bubble-resonant-scattering-1}, by
Lemma~\ref{lem:tail estimate}, it is enough to verify \eqref{eq:ext-Cauchy} with
\begin{align}
A(t)=\Xi_-(t)Q_{\res}\overline{F}(\frac{|D^\parallel|}{\sqrt t}).
\end{align}
We choose a conic cutoff $\chi$, equal to $0$ when $|x| \leq c_0 t$ and $1$ when $|x| \geq c_1 t$. For example we can take $\chi = F(|x|/(c_1t))$. 
Split the nonlinear source in \eqref{eq:ext-Cauchy} as
\begin{align}
\mc{N}(|u(s)|)u(s)=\chi(s) \mc{N}(|u(s)|)u(s)+ (1-\chi(s))\mc{N}(|u(s)|)u(s). \label{eq:Ib-split}
\end{align}
For the contribution from $\chi \mc{N}(|u(s)|)u(s)$, the adjoint endpoint Strichartz estimate and \eqref{eq:one-bubble-exterior-source} give
\begin{align}
\begin{split}
\sup_{T_1\geq t}\|\int_t^{T_1}A(t)e^{-i(t-s)\Delta}\chi(s) \mc{N}(|u(s)|)u(s)\,ds\|_{L^2_x}&\lesssim\|\chi\mc{N}(|u|)u\|_{L^2_sL^{6/5}_x([t,\infty))}
\longrightarrow0.
\end{split} \label{eq:Pplus-ext}
\end{align}

For the contribution from $(1-\chi)\mc{N}(|u|)u$, we will use a non-stationary phase argument in the $\xi^\parallel$ variable. 
The integral kernel of this piece is $-i$ times 
\begin{equation}\label{eq:res integral piece}
 \int\limits_t^{T_1} \! e^{-i(t-s)\Delta_\perp} \!\!\!\!\! \int\limits_{-C\sqrt{t}}^{C\sqrt{t}} \!\!\!\!  e^{i\big((x^\parallel - y^\parallel) \xi^\parallel + (s-t) (\xi^\parallel)^2\big)} \, \Xi_-(x^\parallel, t) q_{\res}(\xi^\parallel) \overline{F}\big(\frac{|\xi^\parallel|}{\sqrt{t}}\big) (1-\chi(y,s))  \mc{N}(|u(y,s)|) u(y,s)    d\xi^\parallel  ds
\end{equation}
where $\Delta_\perp$ is the Laplacian in the plane orthogonal to $e$. (Notice that this operator, and any function of it, commutes with $\Xi_-(t)$ as $\Xi_-(t)$ is a multiplication operator depending only on $x^\parallel$ and $t$.) 
To prove that the $L^2$ norm of this function tends to zero as $t \to \infty$, it is sufficient to show that the operator norm of the function on $L^2(\R)$ with integral kernel, for fixed $y_\perp$, given by
\begin{equation}\label{eq:K1}
K_{y_\perp, t, s}(x^\parallel, y^\parallel) := \int  e^{i\Big((x^\parallel - y^\parallel) \xi^\parallel + (s-t) (\xi^\parallel)^2\Big)} \, \Xi_-(x^\parallel, t) q_{\res}(\xi^\parallel) \overline{F}\big(\frac{|\xi^\parallel|}{\sqrt{t}}\big) (1-\chi(y,s))  \,   d\xi^\parallel 
\end{equation}
is $O((s+t)^{-2})$, uniformly in $y_\perp$. 

The longitudinal phase function $\Phi$ is 
\begin{align}
\Phi^\parallel=(x^\parallel-y^\parallel)\xi^\parallel-(t-s)(\xi^\parallel)^2, \qquad \partial_{\xi^\parallel}\Phi^\parallel=x^\parallel-y^\parallel+2(s-t)\xi^\parallel. \label{eq:one-bubble-resonant-phase}
\end{align}
On the support of $\Xi_-(t)$, $Q_{\res}$, and $1-\chi(s)$, with $s\geq t$, one has
\begin{align}
x^\parallel\leq -c_3(1-\delta_0)t, \qquad y^\parallel\geq -c_1s, \qquad 2\xi^\parallel\leq- c_2(1-\delta_0),
\end{align}
and therefore
\begin{align}
\partial_{\xi^\parallel}\Phi^\parallel\leq-c_3(1-\delta_0)t+c_1s-c_2(1-\delta_0)(s-t).  \label{eq:Pplus-phase-bound}
\end{align}
Since $c_1 < c_2 < c_3$, we have 
\begin{align}
|\partial_{\xi^\parallel}\Phi^\parallel| \geq C_{c_1, c_2, c_3} \Big( (t+s)+|x^\parallel-y^\parallel| \Big). \label{eq:one-bubble-resonant-phase-full}
\end{align}
This non-stationarity is illustrated in the left box of Figure~\ref{fig:nsp}. 
Repeated integration by parts in $\xi^\parallel$ shows that $K_{y_\perp,t,s}$ has operator norm on $L^2(\R)$ bounded by any negative power of $t+s$ as $t \to \infty$, uniformly in $y^\perp$.  
Propositions~\ref{prop:N basic properties} and Proposition~\ref{prop:Linfty} give
\begin{align}
\sup_{s\geq0}\|\mc{N}(|u(s)|)u(s)\|_{L^2_x}<\infty, \label{eq:N-Linfty-L2}
\end{align}
and the transverse free evolution $e^{-i(t-s)\Delta_\perp}$ is treated by $L^2(\R^2)$ unitarity. Putting these estimates into the integral \eqref{eq:res integral piece}, we see that the $L^2$ norm of this expression tends to zero as $t \to \infty$. Combined with \eqref{eq:Pplus-ext}, we have proved  \eqref{eq:one-bubble-resonant-scattering-1}.

Next we prove \eqref{eq:3d-outer-mid-vanishes-1}. We start with the first term, which we write using Duhamel's formula:
\begin{equation}\label{eq:Xi-Duhamel}
\Xi_-(t)Q_{\nr}\overline{F}(\frac{|D^\parallel|}{\sqrt t}) u(t) = 
\Xi_-(t)Q_{\nr}\overline{F}(\frac{|D^\parallel|}{\sqrt t}) \Big( e^{-it\Delta} u(0) -i \int_0^t e^{-i(t-s)\Delta} \mc{N}(|u(s)|) u(s) \, ds \Big). 
\end{equation}
The first term tends to zero in $L^2$, since $Q_{\nr}$ is, by construction, non-resonant with respect to $\Xi_-$; that is, for a scattering solution to end up asymptotically in the support of $\Xi_-(t)$ for large $t$, it should have frequencies with 
$\xi^\parallel \leq -\frac{(1-\delta_0)c_3}{2}$ (apply Lemma~\ref{lem:asymptotic-velocity} with $a = \Xi_-$), yet $Q_{\nr}$ is chosen to vanish on all such frequencies. 

For the integral term, we use the conic cutoff $\chi$ above, recalling that the Strichartz estimate \eqref{eq:one-bubble-exterior-source} is available on the support of $\chi$. 
Then this piece of the integral term is  
\begin{equation}
\Xi_-(t)Q_{\nr}\overline{F}(\frac{|D^\parallel|}{\sqrt t}) \int_0^t e^{-i(t-s)\Delta} \chi(s) \mc{N}(|u(s)|) u(s) \, ds   . 
\end{equation}
We pull the operator $e^{-it\Delta}$ outside the integral, and then thanks to the Strichartz estimate, the integral 
$$
\int_0^t e^{is\Delta} \chi(s) \mc{N}(|u(s)|) u(s) \, ds
$$
converges as $t \to \infty$, to a function $w \in L^2$. That is, 
$$
\int_0^t e^{is\Delta} \chi(s) \mc{N}(|u(s)|) u(s) \, ds = w + o_{L^2}(1), t \to \infty. 
$$
The $w$ term is treated the same way as the $u(0)$ term, and the $o_{L^2}(1)$ remainder is negligible in the limit as $\Xi_-(t)Q_{\nr}\overline{F}(\frac{|D^\parallel|}{\sqrt t}) e^{-it\Delta}$ is a uniformly bounded family of operators as $t \to \infty$. 

So we are reduced to considering the integral term with integrand supported near the bubble. Then we make another non-stationary phase argument similar to the one above. The integral kernel of this piece is $-i$ times 
\begin{equation*}
 \int\limits_0^t \! e^{-i(t-s)\Delta_\perp} \!\!\!\!\! \int\limits_{-C\sqrt{t}}^{C\sqrt{t}}  \!\!\! e^{i\big((x^\parallel - y^\parallel) \xi^\parallel + (s-t) (\xi^\parallel)^2\big)}  \Xi_-(x^\parallel, t) q_{\nr}(\xi^\parallel) \overline{F}\big(\frac{|\xi^\parallel|}{\sqrt{t}}\big) (1-\chi(y,s))  \mc{N}(|u(y,s)|) u(y,s)    d\xi^\parallel  ds.
\end{equation*}
To prove that the $L^2$ norm of this function tends to zero as $t \to \infty$, it is sufficient to show that the operator norm of the function on $L^2(\R)$ with integral kernel, for fixed $y_\perp$, given by
\begin{equation}\label{eq:K2}
K_{y_\perp,t,s}(x^\parallel, y^\parallel) := \int  e^{i\Big((x^\parallel - y^\parallel) \xi^\parallel + (s-t) (\xi^\parallel)^2\Big)} \, \Xi_-(x^\parallel, t) q_{\nr}(\xi^\parallel) \overline{F}\big(\frac{|\xi^\parallel|}{\sqrt{t}}\big) (1-\chi(y,s))  \,   d\xi^\parallel 
\end{equation}
tends to zero faster than $1/t$, uniformly in $y_\perp$. 

\begin{figure}[htbp]
    \centering 
    \includegraphics[width=0.9\textwidth]{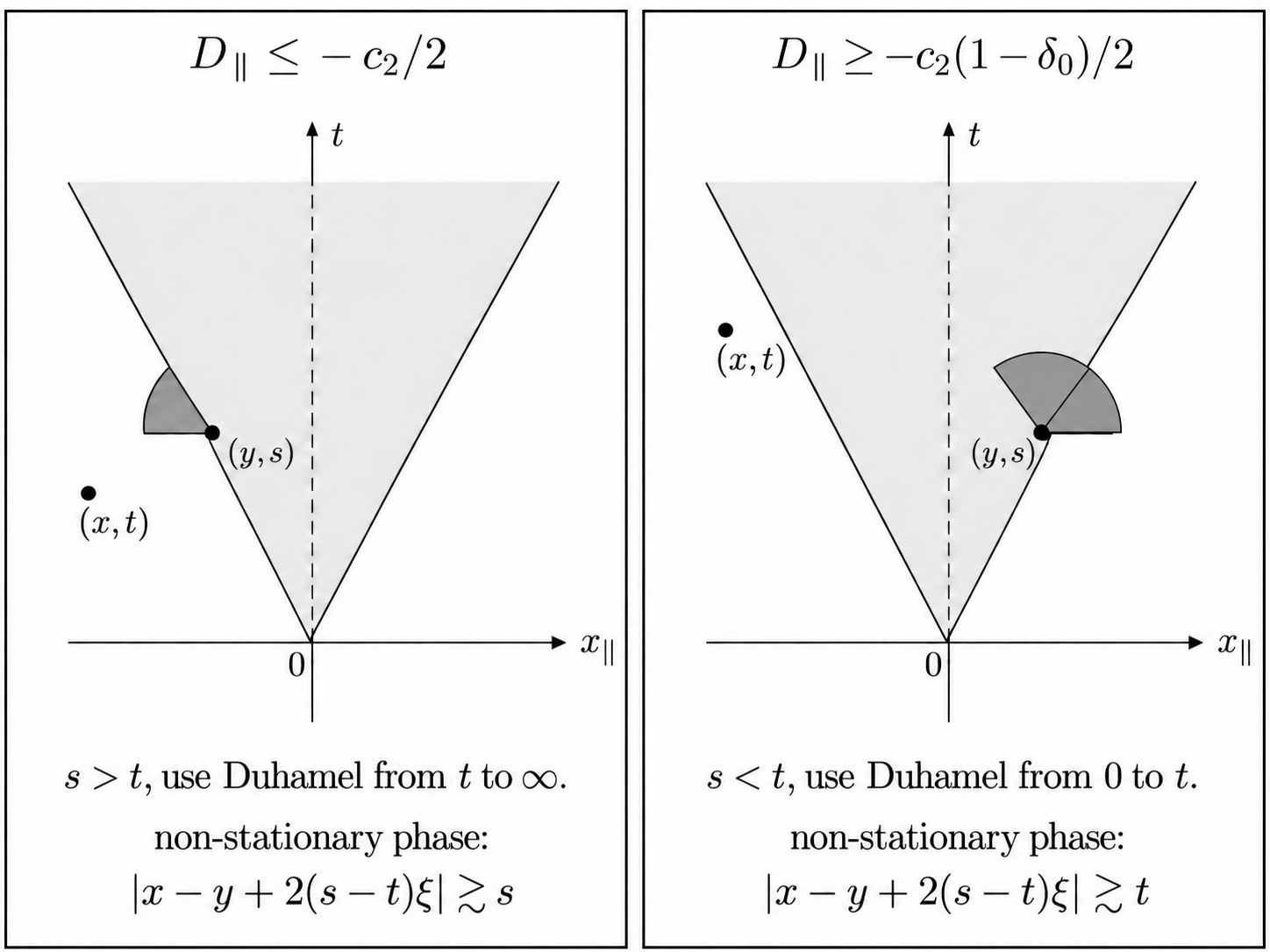} 
    \caption{A visual illustration of the non-stationary phase arguments. The dark shaded circular sector indicates the frequency decomposition: on the left is the range of velocities corresponding to resonant frequencies, and on the right, non-resonant frequencies. Non-stationary phase means that we never have $y-2(s-t)\xi=x$ on the support of the integrand of \eqref{eq:K1}, respectively \eqref{eq:K2}. 
    On the left, we choose $s > t$, corresponding to using Duhamel from $t$ to $\infty$, and we see that we cannot move from $(y, s)$ to $(x, t)$ with  velocity that is a negative multiple of the set of resonant (shaded) directions. On the left, we choose $s < t$, corresponding to using Duhamel from $0$ to $t$, and we see that again, we cannot move from $(y,s)$ to $(x, t)$ with velocity that is a positive multiple of the set of nonresonant (shaded) directions.}
    \label{fig:nsp} 
\end{figure}

We show that this phase $\Phi$ is non-stationary, on the support of the integrand. (This phase is as in \eqref{eq:one-bubble-resonant-phase}.) On the support of the integrand, we have $x^\parallel \leq -c_3(1-\delta_0)t$, $y^\parallel \geq -c_1s$, while from the frequency cutoff $q_{\nr}$ we have $2\xi^\parallel \geq -c_2$. So, 
\begin{equation}
x^\parallel - y^\parallel + 2(s-t) \xi^\parallel \leq -c_3(1-\delta_0)t + c_1 s + (t-s) c_2. 
\end{equation}
From this, as $c_1 < c_2 < c_3$, we have 
$$
|d_{\xi^\parallel}\Phi| \geq C(c_1, c_2, c_3) \Big( |x^\parallel - y^\parallel| + t \Big).
$$
This non-stationarity is illustrated in the right box of Figure~\ref{fig:nsp}.
We can integrate by parts repeatedly as before, showing that this operator has operator norm on $L^2(\R)$ tending to zero faster than any negative power of $t$ as $t \to \infty$. This proves that the first term of \eqref{eq:3d-outer-mid-vanishes-1} tends to zero. The second term is simpler and is treated the same way as the free part of the expression \eqref{eq:Xi-Duhamel}. This completes the proof of the third limiting statement \eqref{eq:3d-outer-mid-vanishes-1}, and hence completes the proof of  Proposition~\ref{prop:one-bubble-directional-scattering}. 

\end{proof}

\begin{remark}\label{rem:half space}
Observe that, in both of the non-stationary phase calculations above, non-stationarity of the phase is valid when $y^\parallel$ is in a half-space; it is not necessary to restrict $y$ to be in the bubble, or even to restrict $y^\parallel$ correspondingly. This observation is useful in the following section. 
\end{remark}

\subsection{Conclusion of the one-bubble argument}
\label{sec:conclusion-onebubble}
We have now shown that $u(t)$ scatters to $e^{-it\Delta} u_+$ on the cone where $\Xi_\pm = 1$ (see Definition~\ref{def:cone scattering}). We next exploit Lemma~\ref{lem:union of cones} to observe that, for any finite number of directions $e_1, \dots, e_N$, we have scattering in the \emph{union} of spacetime cones  $\{ |x_{(j)}^\parallel| \geq c_3 t \}$, where $x_{(j)}^\parallel$ is the longitudinal variable corresponding to direction $e_j$. Recalling that $c_3 < c$, we may choose sufficiently many well-spaced directions $e_j$ such that the union of all these spacetime cones includes $\{ |x| \geq c t \}$. This proves Theorem~\ref{thm:main} for the special case where $C_0 = \{ |x| > c_0 t \}$.

\section{Scattering near the spacetime equator}\label{sec:equator}

In this section our goal is to show, unconditionally, that $u$ scatters in the spacetime cone $\{ |x| > Kt \}$ for sufficiently large $K$; that is, the solution scatters in the far spacetime region $\{|x|\gtrsim Kt\}$ near the spatial infinity of each time slice. In view of the result of Section~\ref{sec:onebubble}, it only remains to show the smallness of $\| u \|_{L^\infty}$ in this spacetime cone. 
So we establish such $L^\infty$-smallness below; in fact we show that the $L^\infty$ norm on $\{ |x| > Kt \}$ tends to $0$ as $K\to\infty$, uniformly in $t\ge1$.

\begin{proposition}\label{prop:equator}
Let $\mc N$ be a nonlinearity in the class of Section~\ref{sec:preliminary-est}, and let $u$ be the solution of \eqref{eq:nls} with $u(0)\in\mc S$; fix $\beta\in(0,1/2)$. Then there are constants $C=C(\mathscr M,\mathscr E)$ and $\theta=\tfrac{2\beta}{2\beta+3}\in(0,1/4)$ such that for $K$ sufficiently large we have 
\begin{equation}\label{eq:equator}
 \lim_{K\to\infty}\|u(t)\|_{L^\infty_{x,t}(\{ |x| \geq Kt \} \cap \{ t \geq 1 \})}=0.
\end{equation}

\end{proposition}

\begin{proof}
\emph{Step 1: $L^2$-smallness in the far region.}
Let $U(t,s)$ be the propagator generated by the self-adjoint, time-dependent Schr\"odinger operator $\Delta+V(t)$ with the real, uniformly bounded potential $V(x,t):=\mc N(|u(x,t)|)$ (bounded by Proposition~\ref{prop:Linfty} and the boundedness of $\mc N(|u|)$). By uniqueness $u(t)=U(t,0)u(0)$, and $U(t,s)$ is unitary on $L^2_x$. Let $F$ and $F_K$ be defined as in \eqref{eq:F-convention} and \eqref{eq:Fscaled}. 
Moving $F_{K} = F_K(|x|/t)$ through the propagator gives
\begin{equation}\label{eq:equator-commutator}
  F_{K}(|x|/t) u(t)=U(t,0)\,F_{K}(|x|/t) u(0)+i\int_0^t U(t,s)\,[\Delta,F_{K}(|x|/t)]\,u(s)\,ds,
\end{equation}
which follows from $\tfrac{d}{ds}\big(U(t,s)F_{K}(|x|/t) U(s,0)\big)=iU(t,s)[\Delta,F_{K}(|x|/t)]U(s,0)$, integrated over $s\in[0,t]$, and the fact that multiplication by the real potential $V(s)$ commutes with $F_{K}(|x|/t)$, so $[\Delta+V(s),F_{K}(|x|/t)]=[\Delta,F_{K}(|x|/t)]$.

Now $[\Delta,F_{K}(|x|/t)]=-2\nabla F_{K}(|x|/t)\cdot\nabla+(\Delta F_{K}(|x|/t))$, where $\nabla F_{K}(|x|/t)$ and $\Delta F_{K}(|x|/t)$ are supported in the annulus $\{|x|\sim Kt\}$ and satisfy $|\nabla F_{K}(|x|/t)|\lesssim (Kt)^{-1}$, $|\Delta F_{K}(|x|/t)|\lesssim (Kt)^{-2}$. Using the uniform bounds $\|\nabla u(s)\|_{L^2_x}+\|u(s)\|_{L^2_x}\le C(\mathscr M,\mathscr E)$ from Proposition~\ref{prop:N basic properties},
\begin{equation}\label{eq:equator-comm-bd}
  \|[\Delta, F_{K}(|x|/t)]u(s)\|_{L^2_x}\lesssim \frac{1}{Kt}\|\nabla u(s)\|_{L^2_x}+\frac{1}{(Kt)^2}\|u(s)\|_{L^2_x}\lesssim \frac{C(\mathscr M,\mathscr E)}{Kt},
\end{equation}
where we used $t\ge 1$ and $K\ge 1$. By unitarity of $U(t,s)$ on $L^2_x$, \eqref{eq:equator-commutator} and \eqref{eq:equator-comm-bd} yield
\begin{align}\label{eq:equator-L2}
\begin{split}
  \| F_{K}(|x|/t) u(t)\|_{L^2_x} \le & \| F_{K}(|x|/t) u(0)\|_{L^2_x}+\int_0^t\|[\Delta, F_{K}(|x|/t)]u(s)\|_{L^2_x}\,ds
  \\ \lesssim & \|u(0)\|_{L^2(|x|\ge Kt/2)}+\frac{C(\mathscr M,\mathscr E)}{K}.   
\end{split} 
\end{align}
Since $u(0)\in L^2$, the tail obeys $\|u(0)\|_{L^2(|x|\ge Kt/2)} \to 0$ as $K \to \infty$. 
Hence
\begin{equation}\label{eq:equator-L2-final}
  \delta_{K}:=\|F(\tfrac{|x|}{Kt})u(t)\|_{L^2_x} \to 0
\end{equation}
as $K \to \infty$.

\emph{Step 2: upgrade to $L^\infty$-smallness.}
By Proposition~\ref{prop:gradient-L2to6-bound} and Corollary~\ref{col: holdercontinue}, $u(t)$ is uniformly bounded in $L^\infty_x$ and uniformly H\"older continuous of order $\beta$: there is a constant $C_2>0$ such that
\begin{equation}\label{eq:equator-holder}
  \|u(t)\|_{L^\infty_x}+\sup_{x\ne y}\frac{|u(x,t)-u(y,t)|}{|x-y|^\beta}\le C_2,\qquad t\ge 1.
\end{equation}
Fix $t\ge1$ and take ${K}\ge 1$ large enough so that $\delta_K \leq 1$, set ${K}':=2{K}$, and let $a:=\|u(t)\|_{L^\infty_x(\{ |x| \geq K' t \})}$. Thus we can find $x_0(t)$ with $|u(x_0, t)| \geq a/2$ and $|x_0|\ge  {K}'t \ge {K}t+1$ (using $t\ge1$). We use the H\"older continuity estimate \eqref{eq:equator-holder} to show that $a$ is controlled by a power of $\delta_K$.  If $a=0$ there is nothing to prove; otherwise set $r:=(a/(4C_2))^{1/\beta}$.
Consider two cases.

$\bullet$ If $r>1$, then $a>4C_2$, and by \eqref{eq:equator-holder} we have $|u(y,t)|\ge a/2-C_2>a/4$ for all $y\in B(x_0,1)$. Since $|x_0|\ge {K}t+1$, we have $B(x_0,1)\subset\{|x|\ge {K}t\}=\{F(|x|/({K}t))=1\}$, so
\[
  \delta_{K}^2\ge\int_{B(x_0,1)}|u(y,t)|^2\,dy\gtrsim a^2 ,
\]
giving $a\lesssim\delta_{K}\le\delta_{K}^{\theta}$ (as $\delta_{K}\le1$ and $\theta\le1$), which proves the claim.

$\bullet$ If $r\le1$, then for $y\in B(x_0,r)$, \eqref{eq:equator-holder} gives $|u(y,t)|\ge a/2-C_2 r^\beta=a/4$. Since $r\le1$ and $|x_0|\ge {K}t+1$, again $B(x_0,r)\subset\{F(|x|/({K}t))=1\}$, so
\begin{equation}\label{eq:equator-ball}
  \delta_{K}^2\ge\int_{B(x_0,r)}|u(y,t)|^2\,dy\ge\frac{a^2}{16}|B(x_0,r)|\gtrsim a^2 r^3\gtrsim C_2^{-3/\beta}a^{2+3/\beta},
\end{equation}
whence $a\lesssim_{C_2}\delta_{K}^{2\beta/(2\beta+3)}=\delta_{K}^{\theta}$ with $\theta$ as in \eqref{eq:equator}, proving the claim.

Combining the claim with \eqref{eq:equator-L2-final}, we have
\begin{equation}\label{eq:equator-Linfty-prime}
  \|u(t)\|_{L^\infty_x(\{ |x| \geq 2Kt \})} \to 0
\end{equation}
for all $t \geq 1$ and $K$ sufficiently large. 
Relabelling $2{K}\mapsto {K}$  turns \eqref{eq:equator-Linfty-prime} into \eqref{eq:equator}. 
\end{proof}

\section{The multi-bubble case}\label{sec:multibubble}
In this section, we treat the case of $N \geq 2$ bubbles in three space dimensions. Thus we have $N$ distinct velocities $v_1, \dots v_N$. Each bubble is a conic set 
$$
\{ (x, t) \mid |x - v_j t| \leq \epsilon t \} \cap \{ t\geq T\}, \quad T \gg 1. 
$$
We will take $\epsilon$ sufficiently small so that the constructions below can be carried out. Roughly speaking, we need the bubbles to be `thin' relative to their separation. 

We will show scattering on the exterior of slightly larger cones 
$$
\{ (x, t) \mid |x - v_j t| \leq 2\epsilon t \} \cap \{ t\geq T\}, \quad T \gg 1 .
$$
We proceed analogously to the one-bubble case, but there is a new difficulty: the non-resonant estimate, that is, the multi-bubble analogue of \eqref{eq:3d-outer-mid-vanishes-1}, does not work as in the previous section. We need a sharper spatial cutoff, which breaks the dilation invariance that our previous cutoffs enjoyed. We begin by proving the required propagation estimate, as well as a functional calculus lemma, the first part of which is well-known, but included for completeness. 

\subsection{Propagation estimate}
\label{sssec:3d-propagation}
Next we prove our propagation estimate. 
For $e\in S^2$, $c\in\R$, and $\sigma\in\{1,-1\}$, set
\begin{align}
b=\frac{\sigma(x\cdot e-ct)}{t^\alpha}, \qquad D_c=D\cdot e-\frac c2, \qquad \frac12<\alpha<1. \label{eq:3d-shifted-variables}
\end{align}

\begin{proposition}[Propagation estimate]
\label{prop:3d-propagation}
Suppose that the support of $F'(b)$ is contained in the region where the Strichartz estimate \eqref{eq:input-L2-2} is valid.  Then
\begin{align}
\int_T^\infty\frac1{t^\alpha}\|\sqrt{F'(b)} D_cu(t)\|_{L^2_x}^2\,dt<\infty. 
\label{eq:3d-propagation-integral}
\end{align}
If $w(t)=e^{-i(t-t_0)\Delta}\phi$ and $D_c\phi\in L^2_x$, then
\begin{align}
\int_T^\infty\frac1{t^\alpha}\|\sqrt{F'(b)} D_c w(t)\|_{L^2_x}^2\,dt\lesssim\|\phi\|_{L^2_x}^2+\|\phi\|_{L^2_x}\|D_c\phi\|_{L^2_x}. 
\label{eq:3d-free-propagation-integral}
\end{align}
\end{proposition}

\begin{remark}\label{rem:sharp cutoff}
It might look surprising at first that we take $\alpha < 1$ instead of $\alpha = 1$, which might look more in the spirit of `conic geometry'. Indeed, $t^{-1}$ would decay faster than $t^{-\alpha}$ for $\alpha < 1$ in \eqref{eq:3d-propagation-integral}, so this looks at first sight to be a more promising candidate for such an estimate. 

However, a moment's thought shows that, even if $u = e^{-it \Delta} w$ is a free solution, then \eqref{eq:3d-propagation-integral} would be infinite for $\alpha = 1$. Indeed, the support of $F'$ would then be a cone, and the typical free solution would have asymptotically nonzero mass in this cone, so the decaying factor $t^{-1}$ would not be sufficient to make the integral converge as $t \to \infty$. 

On the other hand, for $\alpha < 1$, the support of $F'$ is a thin sliver of width $O(t^{\alpha - 1})$ in terms of the conic variable $x/t$, so the mass of a free solution could be expected to decay as $t^{\alpha - 1}$. Moreover, this sliver is centred at $x \cdot e/t = c$, which corresponds, using the dispersion relation for the free equation, to frequencies satisfying $\xi \cdot e = c/2$. The symbol of $D_c$ vanishes on such frequencies, and therefore (taking into account the width of the support of $F'$) provides an \emph{additional} $O(t^{\alpha - 1})$ vanishing factor as $t \to \infty$. Thus, in changing from $t^{-1}$ to $t^{-\alpha}$, we lose one factor of $t^{1-\alpha}$ but we gain $t^{\alpha - 1}$ twice, so overall gaining a decaying power of $t$, leading to integrability of \eqref{eq:3d-propagation-integral} as $t \to \infty$.  
\end{remark}

\begin{proof}
For convenience, we introduce the notation 
\begin{equation}
A \odot B = AB+BA.
\end{equation}
This produces a new symmetric operator from two symmetric operators $A$, $B$. 
Define the $L^2$-pairing 
\begin{align}
a_{c,\sigma}(t)=\frac\sigma2\langle u(t),(F(b)\odot D_c) u(t)\rangle \label{eq:3d-propagation-observable}
\end{align}
where $b$ is as in \eqref{eq:3d-shifted-variables}. We will prove \eqref{eq:3d-propagation-integral} via a propagation estimate. First, 
mass conservation and the uniform $H^1_x$ bound give
\begin{align}
\sup_{t\geq T}|a_{c,\sigma}(t)|\lesssim1. \label{eq:3d-observable-bounded}
\end{align}
We will calculate the time derivative of \eqref{eq:3d-propagation-observable} and show that it is the sum of a positive quantity, and an integrable quantity. The positive quantity will be the integrand of \eqref{eq:3d-propagation-integral}, and the uniform boundedness of the pairing shows that this positive term is integrable in time. 

We compute the time derivative of \eqref{eq:3d-propagation-observable}:
\begin{equation}\begin{aligned}
\frac{d}{dt} a_{c,\sigma}(t) &= \frac{\sigma}{2} \bang{u(t), (\frac{d}{dt}F(b)\odot D_c) u(t)} \\
&+  \frac{\sigma}{2} \bang{u(t), i[\Delta, (F(b)\odot D_c)] u(t)} \\
 - \frac{i\sigma}{2} & \Big(\bang{\mc{N}(|u(t)|) u(t), (F(b)\odot D_c) u } 
 - \bang{u(t), (F(b)\odot D_c) \big( \mc{N}(|u(t)|) u(t) \big)} \Big) .
\end{aligned}\label{eq:derivative expanded}
\end{equation}

We have, since $D_c$ is time-independent, 
\begin{equation}\label{eq:F'(b)}
\frac{d}{dt} \big( F(b) \odot D_c \big)  = \Big( F'(b) \big( - \frac{\sigma c}{t^\alpha} - \frac{\alpha b}{t} \big) \odot D_c \Big). 
\end{equation}
Since $b$ is bounded on the support of $F'$, the second term in the first factor on the right is smaller, in a sense, as $t \to \infty$ compared to the first. We shall see that the first term combines with the commutator term to give the positive operator in the integrand of \eqref{eq:3d-propagation-integral}, and the second term can be absorbed by this positive term. 

For notational convenience, we assume that $e = (1, 0, 0)$, thus $D_c = D_{x_1} - c/2$. 
The commutator $i[\Delta, (F(b)\odot D_c)]$ is (noting that $D_c$ commutes with $\Delta$)
\begin{equation}
\Big( D_{x_1} \odot F'(b) \frac{\sigma}{t^\alpha} \Big) \odot D_c.      
\end{equation} 
If we combine this with the \emph{first} term (only) of the right hand side of \eqref{eq:F'(b)}, then we obtain 
\begin{equation}
\frac{\sigma}{t^\alpha} \Big( D_c \odot F'(b) \Big) \odot D_c.      
\end{equation} 
A short calculation shows that 
\begin{equation}
 \Big( D_c \odot F'(b) \Big) \odot D_c = 4 D_c F'(b) D_c + [[F'(b), D_c], D_c] = 4 D_c F'(b) D_c 
 - \frac{1}{t^{2\alpha}} F'''(b) . 
\end{equation}
Thus the first two lines of the right hand side of \eqref{eq:derivative expanded} give 
\begin{equation}\label{eq:first two lines}
\frac2{t^\alpha}\|\sqrt{F'(b)}D_cu(t)\|_{L^2_x}^2-\frac1{2t^{3\alpha}}\langle u,F'''(b)u\rangle 
-\frac{\alpha\sigma}{2t}\langle u,(bF'(b)\odot D_c)u\rangle.
\end{equation}

The third line we can simplify to 
\begin{equation}
-i\sigma \bang{u(t), F(b) D_{x_1}(\mc{N}(|u(t)|) u(t)}   = -i\sigma \bang{F(b), D_{x_1}(\mc{N}(|u(t)|) |u(t)|^2}. 
\end{equation}
The right hand side is the $D_{x_1}$-derivative of $\tilde{\mc{N}}(|u|)$, for suitable $\tilde N(|u|)$ which is $O(|u|^4)$ (using the uniform $L^\infty$ bound from Proposition~\ref{prop:Linfty}), so we can integrate by parts in $x_1$. This allows us to bound the time-integral of this term by 
\begin{align}
\int_T^\infty\frac1{t^\alpha}\int_{\R^3}F'(b)|u|^4\,dx\,dt\lesssim\|t^{-\alpha}\|_{L^2_t(T,\infty)}\|u\|_{L^\infty_tL^6_x}\|\chi_1 |u|^2u\|_{L^2_tL^{6/5}_x}<\infty \label{eq:3d-nonlinear-remainder-integrable}
\end{align}
(where $\chi_1$ is as in \eqref{eq:input-L2-2}) since $F'(b)$ is supported in the region where we have the Strichartz estimate \eqref{eq:input-L2-2}. 

For $\alpha > 1/3$, the second term of \eqref{eq:first two lines} is also integrable in $t$. We have therefore shown that 
\begin{equation}\label{eq:ac3}
\frac{d}{dt} a_{c,\sigma}(t) = \frac2{t^\alpha}\|\sqrt{F'(b)}D_cu(t)\|_{L^2_x}^2
-\frac{\alpha\sigma}{2t}\langle u,(bF'(b)\odot D_c)u\rangle + g(t),
\end{equation}
where $g$ is integrable in time. Next we note that the second term on the right can be absorbed by the first. Noting that $|b| \leq 1$ on the support of $F'(b)$, we apply Cauchy-Schwarz followed by $a_1 a_2 \leq \eta a_1^2 + \eta^{-1} a_2^2$ to deduce 
\begin{align}
\frac1t|\langle u,(bF'(b)\odot D_c)u\rangle|\leq\frac1{t^\alpha}\|\sqrt{F'(b)}D_cu\|_{L^2_x}^2 + \| F' \|_{L^\infty} t^{\alpha-2}\|u\|_{L^2_x}^2. \label{eq:3d-time-cutoff-absorption}
\end{align}
The first term can be absorbed by the positive term in \eqref{eq:ac3}, and the second is integrable in time as $\alpha < 1$, so this can be absorbed into the $g(t)$ term. We deduce that 
\begin{equation} 
\frac{d}{dt} a_{c,\sigma}(t) \geq \frac1{t^\alpha}\|\sqrt{F'(b)}D_cu(t)\|_{L^2_x}^2 + \tilde g(t),
\end{equation}
where $\tilde g$ is integrable. As $a_{c,\sigma}(t)$ is uniformly bounded, this implies the integrability in time of \eqref{eq:3d-propagation-integral}, completing the proof. 
\end{proof}

We will also have need of the following result. The first, \eqref{eq:asymptotic-velocity}, is well-known, while the second adapts this to our non-conic cutoffs. 

\begin{lemma} 
\label{lem:asymptotic-velocity}
If $a\in C^\infty(\R^3) \cap L^\infty(\R^3)$ and $f\in L^2_x$, then
\begin{align}
\|a(x/t)e^{-it\Delta}f-e^{-it\Delta}a(2D)f\|_{L^2_x}\longrightarrow0 \text{ as } t \to \infty. \label{eq:asymptotic-velocity}
\end{align}
Let $e\in S^2$, $c\in\R$, $\sigma\in\{1,-1\}$, and $0<\alpha<1$. Then
\begin{align}
\mathrm{s}\text{-}\lim_{t\to\infty}e^{it\Delta}F(\frac{\sigma(x\cdot e-ct)}{t^\alpha})e^{-it\Delta}=\mathbf 1_{\{\sigma(2D\cdot e-c)>0\}}. \label{eq:subconic-asymptotic-velocity}
\end{align}
\end{lemma}

\begin{proof}
To prove \eqref{eq:asymptotic-velocity}, we first observe that it suffices to do for a dense subset of $L^2$, as both operators $a(x/t) e^{-it\Delta}$ and $e^{-it\Delta} a(2D)$ are uniformly bounded by $\| a \|_{L^\infty}$. 
We first prove for $f$ such that $\hat f \in C_c^\infty(\R^3)$. 

Let $f$ be such a function, with Fourier transform supported in $\{ |\xi| \leq R \}$. Let $\epsilon > 0$ be given. We can approximate the function $a$ on $B(0, 2R)$ by a polynomial $p(\xi)$, up to a uniform error $\epsilon$. Starting from the identity 
\begin{equation}\label{eq:xD identity}
e^{it\Delta}xe^{-it\Delta}=x+2tD,
\end{equation}
we divide through by $t$ and iterate to find 
\begin{equation}\label{eq:xD identity 2}
e^{it\Delta}p\big( \frac{x}{t} \big) e^{-it\Delta}= p\big( \frac{x}{t} +2D \big).
\end{equation}
Applying the right hand side to $f$ gives 
\begin{equation}
  p\big( \frac{x}{t} +2D \big)f = p(2D) f + O_{L^2}(1/t),  
\end{equation}
since $f$ is Schwartz, hence remains in $L^2$ after multiplication by any monomial in $x$. 
Applying $e^{-it\Delta}$ on the left we find that 
\begin{equation}
p\big( \frac{x}{t} \big) e^{-it\Delta}f = e^{-it\Delta} p(2D) f + O_{L^2}(1/t)  = e^{-it\Delta} a(2D) f  + O(\epsilon) + O_{L^2}(1/t).
\end{equation}
We then apply the asymptotic \eqref{eq:spe} to the left hand side, and we see that the error term in \eqref{eq:spe} contributes $O(1/t)$ in $L^2$ norm, so we may replace $e^{-it\Delta}f$ by its leading asymptotic in \eqref{eq:spe}. In that case, the polynomial $p(x/t)$ again approximates $a(x/t)$ up to a uniform error $\epsilon$. 
We therefore estimate the norm in \eqref{eq:asymptotic-velocity} by $O(\epsilon) + O(1/t)$, which suffices to prove the limiting statement \eqref{eq:asymptotic-velocity}.


To prove \eqref{eq:subconic-asymptotic-velocity}, first take $f$ with $\hat{f}\in C_c^\infty$ as above, and whose support has positive distance from the hyperplane $2\xi\cdot e=c$. On the part where $\sigma(2\xi\cdot e-c)>0$, choose $a\in C^\infty(\R^3)$ which equals one near $2\supp\widehat f$ and is supported where $\sigma(v\cdot e-c)>0$. By \eqref{eq:asymptotic-velocity}, the free solution is asymptotically supported where $a(x/t)=1$, and on this set $\sigma(x\cdot e-ct)\gtrsim t$. Since $\alpha<1$, the cutoff in \eqref{eq:subconic-asymptotic-velocity} is eventually one. The same argument on the part where $\sigma(2\xi\cdot e-c)<0$ shows that the cutoff is asymptotically zero. The result for general $f\in L^2_x$ follows by first removing a small frequency neighbourhood of $2\xi\cdot e=c$ and then using density, since that hyperplane has measure zero.
\end{proof}

\subsection{Scattering in longitudinal slabs}

The first step is to choose a direction $e \in S^2$ such that each velocity $v_j$ has a distinct component in the $e$ direction; we will call such a direction \emph{suitable}. We will write this component $v_j^\parallel$. Similarly we will write $x^\parallel$, $\xi^\parallel$ for the components of $x$ and $\xi$ in this direction. We will assume that $v_1^\parallel < v_2^\parallel < \dots < v_N^\parallel$. We will also choose constants $\epsilon_j$, $j = 1 \dots 4$, such that $\epsilon = \epsilon_0 < \epsilon_1 < \epsilon_2 < \epsilon_3 < \epsilon_4 < (2/\sqrt{3}) \epsilon$, similar to the previous section, and take $\delta_0$ in \eqref{eq:F-convention} small compared to $\epsilon_{j+1}/\epsilon_j - 1$, $j = 0 \dots 3$.

Similar to the previous section, we choose a conic cutoff $\chi$, equal to $0$ on the cones $\{ |x - v_j t| \leq \epsilon_0 t \}$, and $1$ outside the slightly larger cones $\{ |x - v_j t| \leq \epsilon_1 t \}$. 
Using only the component $x^\parallel$ and $t$, we define conic cutoffs $\Xi_0, \dots, \Xi_N$ supported `between' the bubbles. Here, $\Xi_0$ is supported to the left (in terms of the coordinate $x^\parallel$) of all bubbles, $\Xi_j$ is supported between the $j$th and $(j+1)$st bubble (again in terms of the coordinate $x^\parallel$), and $\Xi_N$ is supported to the right of the $N$th bubble. More precisely, we define 
\begin{equation}\begin{gathered}
\Xi_0 = F \Big( -\frac{x^\parallel - v_1^\parallel t }{\epsilon_4 t} \Big), \\
\Xi_{j} = F \Big( \frac{x^\parallel - v_j^\parallel t}{\epsilon_4 t}  \Big) F \Big( -\frac{x^\parallel - v_{j+1}^\parallel t}{\epsilon_4 t} \Big), \quad j = 1 \dots N-1,  \\
\Xi_N = F \Big( \frac{x^\parallel - v_N^\parallel t }{\epsilon_4 t} \Big).
\end{gathered}\end{equation}

For each cutoff $\Xi_j$, we define a splitting of longitudinal frequency $\xi^\parallel$ into resonant and nonresonant regions. We define 
\begin{equation}\begin{gathered}
Q_{\res, 0} = F \Big( - \frac{2D^\parallel - v_1^\parallel}{\epsilon_2} \Big), \\
Q_{\res, j} = F \Big(  \frac{2D^\parallel - v_j^\parallel}{\epsilon_2} \Big)F \Big( - \frac{2D^\parallel - v_{j+1}^\parallel}{\epsilon_2} \Big), \quad j = 1 \dots N-1, \\
Q_{\res, N} = F \Big(  \frac{2D^\parallel - v_N^\parallel}{\epsilon_2} \Big), \\
Q_{\nr, j} = \Id - Q_{\res, j}. 
\end{gathered}\end{equation}

Then we prove the analogue of Proposition~\ref{prop:one-bubble-directional-scattering}:

\begin{proposition}
\label{prop:multi-bubble-directional-scattering}
We have scattering on each spacetime cone $\{ \Xi_j = 1 \}$:
\begin{align}
\|\Xi_j(t)(u(t)-e^{-it\Delta}u_+)\|_{L^2_x}\longrightarrow0. \label{eq:multi-bubble-directional-scattering}
\end{align}
Moreover, the following statements hold, for $j = 0 \dots N$, which
 collectively imply \eqref{eq:multi-bubble-directional-scattering}. 
\begin{align}
  \|F(\frac{|D^\parallel|}{\sqrt t})u(t)\|_{L^2_x}+\|F(\frac{|D^\parallel|}{\sqrt t})e^{-it\Delta}u_+\|_{L^2_x}&\longrightarrow 0. \label{eq:IIb-vanish-1-multi} \\
 \|\Xi_j(t)Q_{\res, j}\overline{F}(\frac{|D^\parallel|}{\sqrt t})(u(t)-e^{-it\Delta}u_+)\|_{L^2_x}&\longrightarrow0. \label{eq:multi-bubble-resonant-scattering-1} \\
 \|\Xi_j(t)Q_{\nr, j}\overline{F}(\frac{|D^\parallel|}{\sqrt t}) u(t)\|_{L^2_x} + \|\Xi_j(t)Q_{\nr, j}\overline{F}(\frac{|D^\parallel|}{\sqrt t})e^{-it\Delta} u_+ \|_{L^2_x}&\longrightarrow0. \label{eq:3d-outer-mid-vanishes-1-multi}
\end{align}
\end{proposition}

\begin{proof}
The first statement, \eqref{eq:IIb-vanish-1-multi}, is just a restatement of \eqref{eq:IIb-vanish-1}. 

The second statement, with the `resonant' terms, is proved in exactly the same way as was \eqref{eq:one-bubble-resonant-scattering-1}. This is literally true for $j= 0, N$ where the resonant set has the same form (apart from the nonzero velocity of the bubble, which is inconsequential). For, while the support of $1-\chi$ is bigger in this case (it consists of $N$ cones instead of only one), in fact, as noted in Remark~\ref{rem:half space}, the proof of non-stationarity of the phase is valid not only when the $y$ coordinate is in the bubble, but also when it is on the opposite side of the bubble compared to the $x$ coordinate. This observation also shows that the proof adapts immediately to the intermediate cases, $j = 1 \dots N-1$. 

However, for the nonresonant terms in \eqref{eq:3d-outer-mid-vanishes-1-multi}, this is no longer true. Essentially this is because the resonant set is smaller in the intermediate cases ($j = 1 \dots N-1$), compared to the outer cases ($j = 0, N$) or to the one-bubble case, because in the intermediate cases $\xi^\parallel$ has a two-sided bound on the support of $q_{\res, j}$, instead of just a one-sided bound. This makes the resonant case easier, but the non-resonant case more difficult. We need an additional idea. 

The additional idea is to introduce an additional, non-conic cutoff $B_j$ on the right-hand side of the nonresonant cutoff $Q_{\nr, j}$, which cuts off more sharply than does a conic cutoff. It allows us to use Proposition~\ref{prop:3d-propagation}; see Remark~\ref{rem:sharp cutoff} for a heuristic discussion on the benefit of the sharper cutoff. We define
\begin{equation}\begin{gathered}
B_0 = F \Big( -\frac{x^\parallel - (v_1^\parallel - \epsilon_3) t }{t^\alpha} \Big), \\
B_{j} = F \Big( \frac{x^\parallel - (v_j^\parallel + \epsilon_3) t}{t^\alpha}  \Big) F \Big( -\frac{x^\parallel - (v_{j+1}^\parallel - \epsilon_3) t}{t^\alpha} \Big), \quad j = 1 \dots N-1,  \\
B_N = F \Big( \frac{x^\parallel - (v_N^\parallel + \epsilon_3) t }{t^\alpha} \Big).
\end{gathered}\end{equation}
We collect some elementary properties of $B_j$: 

\begin{lemma}\label{lem:B properties}
\hfill
\begin{itemize} 
\item For all $0 \leq j \leq N$, and for all sufficiently large times, $B_j$ is identically $1$ on the support of $\Xi_j$, and the distance between the supports of $\Xi_j$ and $\nabla B_j$ is bounded below by $c t$ as $t \to \infty$. 
\item Each $B_j$ is supported in the region where the Strichartz estimate \eqref{eq:input-L2-2} holds. (This is the reason for offseting the longitudinal velocities $v_j^\parallel$ by $\epsilon_3 > \epsilon_0$ in the definition of $B_j$.) 
\item Assume that $\epsilon_1$ is sufficiently small, and $t$ is sufficiently large. Then for $1 \leq j \leq N-1$, the gradient $\nabla B_j$ is equal to 
\begin{equation}
\nabla \Big( F \Big( \frac{x^\parallel - (v_j^\parallel + \epsilon_3) t}{t^\alpha}  \Big) \Big) + \nabla \Big( F \Big( -\frac{x^\parallel - (v_{j+1}^\parallel - \epsilon_3) t}{t^\alpha} \Big) \Big),
\end{equation}
i.e. each of the two $F$ factors is identically $1$ on the support of the gradient of the other. 
\item Given the dispersion relation $x/t = 2\xi$ for the free Schr\"odinger equation, the function $\nabla B_j$ has disjoint support from $Q_{\nr, j}$; that is, $\nabla B_j$ is resonant in terms of the decomposition $\Id = Q_{\res, j} + Q_{\nr, j}$. 
\end{itemize}
\end{lemma}
We omit the proofs of these properties, which are all easily checked, but mention that the first property holds since $\epsilon_3 < \epsilon_4$, while the second holds since $\epsilon_3 > \epsilon_1$ and the fourth holds since $\epsilon_3 > \epsilon_2$. 

We return to the proof of \eqref{eq:3d-outer-mid-vanishes-1-multi}. We write 
\begin{equation}
\Xi_j(t)Q_{\nr, j}\overline{F}(\frac{|D^\parallel|}{\sqrt t}) u(t) = \Xi_j(t)Q_{\nr, j}\overline{F}(\frac{|D^\parallel|}{\sqrt t})B_j(t) u(t) + \Xi_j(t)Q_{\nr, j}\overline{F}(\frac{|D^\parallel|}{\sqrt t}) (1 - B_j(t)) u(t). 
\end{equation}
It is not difficult to see that the second term on the right hand side tends to zero in $L^2(\R^3)$ as $t \to \infty$. We use the first property in Lemma~\ref{lem:B properties}, to see that the distance between the supports of $\Xi_j(t)$ and $(1 - B_j(t))$ is bounded below by $c t$ as $t \to \infty$. Since the intermediate operator $Q_{\nr, j}\overline{F}(\frac{|D^\parallel|}{\sqrt t})$ is pseudodifferential (in $x^\parallel$), its Schwartz kernel is smooth away from the partial diagonal in $x^\parallel$, and rapidly decreasing as a function of the distance from it. (Notice that although the operator $\overline{F}(\frac{|D^\parallel|}{\sqrt t})$ depends on $t$, it is in the semiclassical calculus with parameter $h = t^{-1/2}$. As such, its Schwartz kernel decays away from the diagonal uniformly --- indeed, more and more rapidly --- as $t \to \infty$. An alternative argument is to remove this factor, which makes at most $O(t^{-1/2})$ difference to the $L^2$ norm thanks to \eqref{eq:IIb-vanish}.) Schur's test shows that the operator $\Xi_j(t)Q_{\nr, j}\overline{F}(\frac{|D^\parallel|}{\sqrt t}) (1 - B_j(t))$ tends to zero rapidly as $t \to \infty$, so we can discard this term and analyze the first term, with the $B_j$ factor. 

Next, to show that the first term on the right hand side tends to zero, it certainly suffices to show that 
\begin{equation*}
    \big\| Q_{\nr, j} B_j(t) u(t) \big\|_{L^2(\R^3)} \to 0
\end{equation*}
(remembering that the two frequency cutoffs $Q_{\nr, j}$ and $\overline{F}(\frac{|D^\parallel|}{\sqrt t})$ commute), so we turn to proving \eqref{eq:wts}. Clearly, it suffices to show that 
\begin{equation}\label{eq:wts}
    \big\| e^{it\Delta} Q_{\nr, j} B_j(t) u(t) \big\|_{L^2(\R^3)} \to 0.
\end{equation}
This will follow from the following Lemma.

\begin{lemma} \label{lem:nonres multi}
(i) The family of functions $e^{it\Delta} Q_{\nr, j} B_j(t) u(t)$ converges to zero weakly in $L^2$, as $t \to \infty$. 

(ii) $e^{it\Delta} Q_{\nr, j} B_j(t) u(t)$ has a strong limit in $L^2$, as $t \to \infty$. 
\end{lemma}

\begin{proof}[Proof of Lemma]
(i) Since $Q_{\nr,j}$ commutes with $e^{is\Delta}$,
\begin{align}
e^{it\Delta}Q_{\nr,j}B_j(t)u(t)=Q_{\nr,j}(e^{it\Delta}B_j(t)e^{-it\Delta})(e^{it\Delta}u(t)). \label{eq:3d-weak-factorisation}
\end{align}
Lemma~\ref{lem:asymptotic-velocity} gives
\begin{align}
\mathrm{s}\text{-}\lim_{t\to\infty}e^{it\Delta}B_j(t)e^{-it\Delta}=\mathbf 1_{\{2D_\parallel \in I_j\}}, \label{eq:3d-outer-conjugated-limit}
\end{align}
where $I_j$ is the interval 
\begin{equation}
    \begin{cases}
  (-\infty, v_1^\parallel - \epsilon_3] \text{ for } j=0, \\
  [v_j^\parallel + \epsilon_3, v_{j+1}^\parallel - \epsilon_3] \text{ for } 1 \leq j \leq N-1, \\
  [v_N^\parallel + \epsilon_3, \infty) \text{ for } j=N. 
    \end{cases}
\end{equation}
On the other hand, $Q_{\nr, j}$ is supported where $2D^\parallel$ is in the set
\begin{equation}
    \begin{cases}
  [v_1^\parallel - \epsilon_2, \infty)      \text{ for } j=0, \\
  (-\infty, v_j^\parallel + \epsilon_2] \cup [v_{j+1}^\parallel - \epsilon_2, \infty) \text{ for } 1 \leq j \leq N-1, \\
  (-\infty, v_N^\parallel + \epsilon_2] \text{ for } j=N.
    \end{cases}
\end{equation}
In every case, the limiting projection is disjoint from the frequency support of $Q_{\nr, j}$. Hence both $Q_{\nr,j} e^{it\Delta}B_j(t)e^{-it\Delta}$ and its adjoint converge strongly to zero. Therefore, for any $\phi \in L^2$, we have
\begin{equation}\begin{gathered}
    \bang{\phi, e^{it\Delta}Q_{\nr,j}B_j(t)u(t)} = 
    \bang{e^{it\Delta} B_j(t) e^{-it\Delta} Q_{\nr, j} \phi, e^{it\Delta} u(t)}
\end{gathered}
\end{equation}
and the left hand side of the inner product tends to zero strongly, while the right hand side is uniformly bounded. This inner product therefore converges to zero for every $\phi$, proving (i). 

(ii) We use Cook's method and show that the integral of the derivative has tail tending to zero in $L^2$. 
For finite $T_1>t$, differentiate $e^{is\Delta}Q_{\nr,j}B_j(s)u(s)$ on $[t,T_1]$. Since $Q_{\nr,j}$ commutes with the free group,
\begin{align}
\begin{split}
\partial_s(e^{is\Delta}Q_{\nr,j}B_j(s)u(s)) = &e^{is\Delta}Q_{\nr,j}(\partial_sB_j(s)+i[\Delta,B_j(s)])u(s)
\\ & -ie^{is\Delta}Q_{\nr,j}B_j(s)\mc{N}(|u(s)|)u(s).     
\end{split}
\label{eq:3d-L-derivative-expansion}
\end{align}
Integrating this identity on $[t,T_1]$ gives
\begin{align*}
\begin{split}
e^{it\Delta}Q_{\nr,j}B_j(t)u(t)=&e^{iT_1\Delta}Q_{\nr,j}B_j(T_1)u(T_1)-\int_t^{T_1}e^{is\Delta}Q_{\nr,j}\Big(\partial_sB_j(s)
+i[\Delta,B_j(s)]\Big)u(s)\,ds \\
&+i\int_t^{T_1}e^{is\Delta}Q_{\nr,j}B_j(s)\mc{N}(|u(s)|)u(s)\,ds.
\end{split} \label{eq:3d-L-terminal-duhamel}
\end{align*}
We start with the last term on the right. Because the support of $B_j$ lies in a conic region in which the Strichartz estimate holds, according to the second item of Lemma~\ref{lem:B properties}, we can bound this term in norm by the $L^2([t, \infty)_s; L^{6/5}(\R^3))$ norm of $B_j(s)\mc{N}(|u(s)|)u(s)$, which tends to zero as $t \to \infty$. 

For the remaining terms, we pair with an arbitrary $\phi \in L^2$.
Then part (i) of the Lemma shows, as $T_1 \to \infty$, we have 
\begin{align*}
   \bang{\phi, e^{iT_1\Delta}Q_{\nr,j}B_j(T_1)u(T_1)} \to 0.
\end{align*}

So it suffices to show, uniformly in $T_1$:
\begin{equation}\label{eq:stronglimit}
\sup_{\| \phi \|_{L^2} = 1} 
\Big|\bang{\phi,\int_t^{T_1}e^{is\Delta}Q_{\nr,j}\Big(\partial_sB_j(s)
+i[\Delta,B_j(s)]\Big)u(s)\,ds}\Big| \to 0.
\end{equation}

In the course of the proof of Proposition~\ref{prop:3d-propagation} we computed the Heisenberg derivative of $B_j$. Using item 3 of Lemma~\ref{lem:B properties} we find that, with 
$$
b_j^{\pm} = \pm \frac{x^\parallel - (v_j^\parallel \pm \epsilon_3) t }{t^\alpha}, 
$$
\begin{equation}\label{eq:Heisenberg1}
\partial_sB_0(s)+i[\Delta,B_0(s)]= -\frac1{s^\alpha}\big(D_\parallel - (v_1^\parallel - \epsilon_3)/2 \big)\odot F'(b_1^-)-\frac\alpha s b_1^-F'(b_1^-),
\end{equation}
while for $1 \leq j \leq N-1$ we have 
\begin{multline}\label{eq:Heisenberg2}
\partial_sB_j(s)+i[\Delta,B_j(s)]=    \frac1{s^\alpha}\big(D_\parallel - (v_j^\parallel + \epsilon_3)/2 \big)\odot F'(b_j^+)-\frac\alpha s b_j^+F'(b_j^+) \\ -\frac1{s^\alpha}\big(D_\parallel - (v_{j+1}^\parallel - \epsilon_3)/2 \big)\odot F'(b_{j+1}^-)-\frac\alpha s b_{j+1}^-F'(b_{j+1}^-),
\end{multline}
and for $j=N$
\begin{equation}\label{eq:Heisenberg3}
 \partial_sB_N(s)+i[\Delta,B_N(s)]=    \frac1{s^\alpha}\big(D_\parallel - (v_N^\parallel + \epsilon_3)/2 \big)\odot F'(b_N^+)-\frac\alpha s b_N^+F'(b_N^+). 
\end{equation}
We further observe that, due to the fourth item in Lemma~\ref{lem:B properties}, we can write 
\begin{equation}\begin{aligned}
Q_{\nr, 0} &= \tilde Q^-_{\nr, 0} \big(D^\parallel - (v_1^\parallel - \epsilon_3)/2 \big); \\
Q_{\nr, j} &= \tilde Q^+_{\nr, j} \big(D^\parallel - (v_j^\parallel + \epsilon_3)/2\big), \quad j = 1 \dots N-1; \\
Q_{\nr, j} &= \tilde Q^-_{\nr, j} \big(D^\parallel - (v_{j+1}^\parallel - \epsilon_3)/2\big), \quad j = 1 \dots N-1; \\
Q_{\nr, N} &= \tilde Q^+_{\nr, N} \big(D^\parallel - (v_N^\parallel + \epsilon_3)/2\big)
\end{aligned}\end{equation}
for bounded Fourier multipliers $\tilde Q_{\nr, j}^\pm$, since the symbols of $Q_{\nr, j}$ vanish at the vanishing set of the corresponding first order differential operator. Combined with \eqref{eq:Heisenberg1} --- \eqref{eq:Heisenberg3} we can express the operator $Q_{\nr,j}(\partial_sB_j(s)
+i[\Delta,B_j(s)])$ as a sum of either one term (for $j = 0$ or $N$) or two terms (for $1 \leq j \leq N-1$) of the form 
\begin{equation}\label{eq:pos term}
\tilde Q^\pm_{\nr, j} \Big( \frac1{s^\alpha} (D^\parallel - c/2) F'(b) (D^\parallel - c/2) + \frac{\alpha}{s} F'(b) (D^\parallel - c/2) + O(s^{-2\alpha}) (D^\parallel - c/2) \Big)
\end{equation}
where $c$ is of the form $v_k^\parallel \pm \epsilon_3$, $k = j$ or $j+1$, and $b$ is the corresponding $b_k^\pm$. The $O(s^{-2\alpha})$ comes from commuting the derivative through the $F'$ factor. We recognize the first term in \eqref{eq:pos term} as the square of the operator that appears in the propagation estimate, \eqref{eq:3d-propagation-integral}. Since $\alpha > 1/2$, the $O(s^{-2\alpha})$ term is integrable in $s$ so contributes a term that goes to zero as $t \to \infty$.   Returning to the integral term in \eqref{eq:stronglimit}, we need to estimate 
\begin{equation}\label{eq:expand}
   \int_t^{T_1} \!\! \bang{e^{-is\Delta} \tilde Q^\pm_{\nr, j}\phi, \Big( s^{-\alpha}(D^\parallel - c/2) F'(b) (D^\parallel - c/2) + \frac{\alpha}{s} F'(b) (D^\parallel - c/2) + O(s^{-2\alpha}) (D^\parallel - c/2) \Big) u(s)}  ds.
\end{equation}
The first, positive term is estimated by moving the square root of the operator to the left hand side and using Cauchy-Schwarz:
\begin{multline}
   \bigg| \int_t^{T_1} \bang{e^{-is\Delta} \tilde Q^\pm_{\nr, j}\phi,  s^{-\alpha}(D^\parallel - c/2) F'(b) (D^\parallel - c/2)   u(s)} \, ds \bigg|  \\ \leq 
  \bigg(  \int_t^{T_1} s^{-\alpha} \| \sqrt{F'} (D^\parallel - c/2) e^{-is\Delta} \tilde Q^\pm_{\nr, j}\phi \|_{L^2}^2 \, ds \bigg)^{1/2} 
  \bigg(  \int_t^{T_1} s^{-\alpha} \| \sqrt{F'} (D^\parallel - c/2) u(s) \|_{L^2}^2 \, ds \bigg)^{1/2} 
\end{multline}
The first bracket on the right is controlled by $C (\| \tilde Q^\pm_{\nr, j}\phi \|_{L^2} + \| D_c \tilde Q^\pm_{\nr, j} \phi \|_{L^2}) $ according to \eqref{eq:3d-free-propagation-integral}. Since both $\tilde Q^\pm_{\nr, j}$ and $D_c \tilde Q^\pm_{\nr, j} = Q_{\nr, j}$ are bounded Fourier multipliers, this is bounded by $C \| \phi \|_{L^2}$. According to \eqref{eq:3d-propagation-integral}, the second bracket is finite and tends to zero as $t \to \infty$. Taking the supremum over all $\phi \in L^2$ with $\|\phi\|_{L^2}=1$, this tends to zero, uniformly in $T_1$. 

We now consider the second term in \eqref{eq:expand}. This is estimated using Cauchy-Schwarz again: 
\begin{multline}
  \bigg| \int_t^{T_1} \bang{e^{-is\Delta} \tilde Q^\pm_{\nr, j}\phi,  \frac{\alpha}{s} F'(b) (D^\parallel - c/2)   u(s)} \, ds \bigg|  \\ \leq 
  \bigg(  \int_t^{T_1} s^{-2 + \alpha} \|  e^{-is\Delta} \tilde Q^\pm_{\nr, j}\phi \|_{L^2}^2 \, ds \bigg)^{1/2} 
  \bigg(  \int_t^{T_1} s^{-\alpha} \| \sqrt{F'} (D^\parallel - c/2) u(s) \|_{L^2}^2 \, ds \bigg)^{1/2}    
\end{multline}
It is clear that the first factor on the RHS is uniformly bounded as the power $s^{-2 + \alpha}$ is integrable in time. The second factor is controlled by \eqref{eq:3d-propagation-integral}. So the RHS goes to zero, uniformly in $T_1$ and uniformly for all $\phi$ with $\| \phi \|_{L^2} = 1$. This completes the proof of part (ii) of the Lemma. 
\end{proof}
The combination of (i) and (ii) of Lemma~\ref{lem:nonres multi} proves \eqref{eq:wts}, since the weak and strong limits coincide. This is the first part of showing \eqref{eq:3d-outer-mid-vanishes-1-multi}. The second part is proved using the same argument applied to the free flow instead of the nonlinear flow \eqref{eq:nls}, and is slightly easier as there is no nonlinear term to estimate; we omit the details. This completes the proof of Proposition~\ref{prop:multi-bubble-directional-scattering}. 
    
\end{proof}

\begin{figure}[htbp]
    \centering 
    \includegraphics[width=0.8\textwidth]{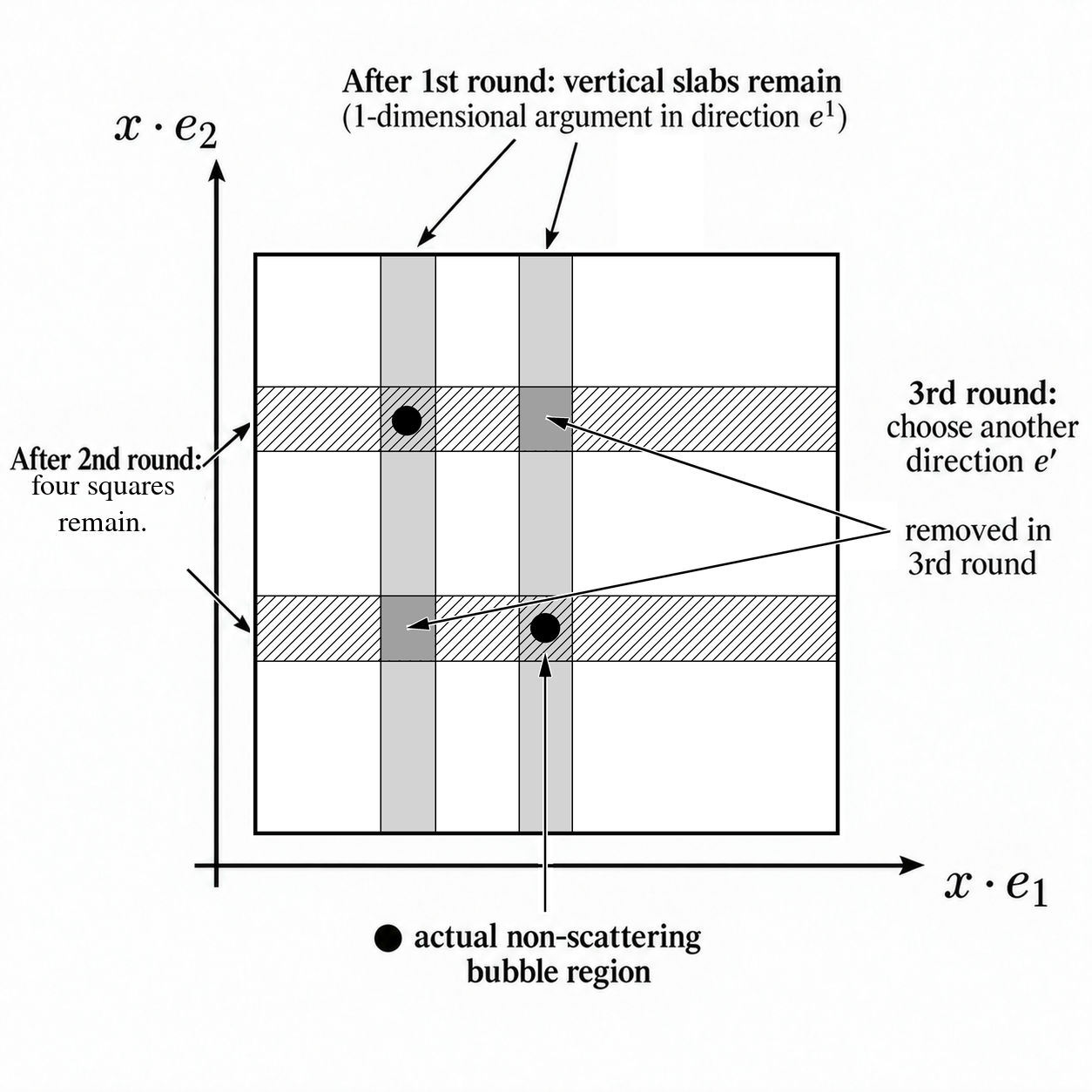} 
    \caption{An illustration of the final step of the proof in two dimensions. After using two orthogonal directions $e^1$ and $e^2$, the non-scattering region is shown to be a subset of 4 squares, of which two contain the velocities of the two bubbles, indicated with a bullet, and two are spurious. We can exclude the two spurious squares by taking a third direction $e^*$ so that the spurious cubes are disjoint from the $e^*$-slabs generated by the two bubbles. For example here we could take $e^* = 1/\sqrt{2} (e^1 + e^2)$.}  
    \label{fig:cubes} 
\end{figure}

\subsection{Proof of Theorem~\ref{thm:main}} Proposition~\ref{prop:multi-bubble-directional-scattering} shows scattering outside $N$ thin slabs located, in the longitudinal direction, near the $N$ bubbles. We call this set, outside of which we have scattering, the $e$-excluded set. By taking other suitable directions and applying Proposition~\ref{prop:multi-bubble-directional-scattering}, we increase the size of the region on which we know we have scattering, thanks to Lemma~\ref{lem:union of cones}. We show that we can show scattering for all velocities $v = x/t$ satisfying $|v - v_j| \geq 2\epsilon$ for all $j = 1 \dots N$.

Let us begin by observing that the set of `unsuitable' directions $\tilde e$, that is, directions which do \emph{not} separate the $N$ velocities $v_j$, in the sense that $v_j \cdot \tilde e$ are not all distinct, consists of at most $N(N-1)/2$ great circles contained in $S^2$.  Indeed, there are $N(N-1)/2$ pairs of velocities, and for each pair, the set of directions that do not separate the pair is the set of directions orthogonal to the difference of the two velocities, which is a great circle. Thus the set of suitable directions, that \emph{do} separate all the velocities in the sense above, is the complement of a codimension 1 analytic set, in particular open and dense. 

It's easy to see that we can find 3 mutually orthogonal suitable directions. After a rotation of coordinates we may, without loss of generality, suppose these are the three standard basis vectors. Let us denote these $e^1, e^2, e^3$. Then applying the argument above to each $e^i$, we have scattering except for those directions $x/t$ which are in the intersection of the $e^j$-excluded sets, for $j = 1 \dots 3$.  This is a union of $N^3$ cubes of side length $2\epsilon_4$. See Figure~\ref{fig:cubes} for an illustration of this setup in two dimensions.  
Of these $N^3$ cubes, $N$ contain one of the velocities $v_j$, and are contained in the ball of radius $\sqrt{3}\epsilon_4 < 2 \epsilon$ around $v_j$, provided we choose $\epsilon_4 < (2/\sqrt{3}) \epsilon$. The rest of the cubes are `spurious' and for each spurious cube $Q$, we look for a direction $e^*$ which shows that scattering holds for $x/t \in Q$. 

Provided that $\epsilon$ is small relative to the difference of velocities, $|v_i - v_j|$, we can always find such a $e^*$. Indeed, consider adding the centre $v_{N+1}$ of $Q$ to the original $N$ velocities. Then we can find a direction $e^*$ such that $v_j \cdot e^*$ are distinct, for $1 \leq j \leq N+1$, by the argument above. Using $e^*$ as the longitudinal direction, the $e^*$-slab meeting $Q$ will be disjoint from the $e^*$-slab meeting any of the velocities $v_1 \dots v_N$, provided $\epsilon$ is small enough. Thus Proposition~\ref{prop:multi-bubble-directional-scattering}, with $e^*$ as the longitudinal direction, shows that scattering occurs on $Q$.

As this argument applies to each spurious cube, this shows scattering occurs outside the $N$ cubes of side length $2\epsilon_4$ around the velocities $v_j$. We have already observed that these cubes are contained in the ball of radius $2\epsilon$ about $v_j$, so the proof is complete. 

\section{Declarations}

\emph{Corresponding author.} Andrew Hassell is the corresponding author. 

\emph{Funding declaration.} The authors gratefully acknowledge support from the Australian Research Council via Laureate Fellowship grant FL220100072. 

\emph{Author contribution declaration.} All authors were equally involved in creating the mathematical content and in the writing of the manuscript.

\emph{AI declaration.} The authors acknowledge the use of AI to assist with language editing and phrasing refinement during the initial stages (Claude, Codex) and checking accuracy during the final stages (ChatGPT).

\emph{Data availability declaration.} No data was produced during this research project.

\bibliographystyle{amsplain}
\bibliography{refs}

\end{document}